\documentclass[letterpaper, 10pt, conference]{ieeeconf} 
\IEEEoverridecommandlockouts    
\usepackage{verbatim}
\usepackage{tikz}
\usetikzlibrary{positioning,shapes,arrows}
\usetikzlibrary{shapes.geometric}
    \usepackage{amsmath,amssymb,amsfonts
    }  
    \usepackage{graphicx,wrapfig}
\let\labelindent\relax
\usepackage{enumitem}
\usepackage{xcolor}
    \def\BibTeX{{\rm B\kern-.05em{\sc i\kern-.025em b}\kern-.08em
        T\kern-.1667em\lower.7ex\hbox{E}\kern-.125emX}}
  \usepackage[pdftex, pdfborderstyle={/S/U/W 0}]{hyperref}

\newtheorem{exe}{Example}
\newtheorem{corol}{Corollary}
\newtheorem{ass}{Assumption}
\newtheorem{defin}{Definition}
\newtheorem{prope}{Property}
\newtheorem{prob}{Problem}
\newtheorem{cla}{Claim}
\newtheorem{rem}{Remark}
\newtheorem{lem}{Lemma}
\newtheorem{prop}{Proposition}
\newtheorem{thm}{Theorem}
\newtheorem{fct}{Fact}
\newenvironment{lemma}{\begin{lem}}{\hfill $\square$ \end{lem}}

\newenvironment{remark}{\begin{rem} \rm}{\hfill $\bullet$ \end{rem}}
\newenvironment{assumption}{\begin{ass}}{\hfill $\bullet$ \end{ass}}
\newenvironment{theorem}{\begin{thm}}{\hfill $\square$ \end{thm}}

\newenvironment{problem}{\begin{prob}}{\hfill $\bullet$ \end{prob}}

\usepackage{amssymb}
\usepackage{amsmath}

\usepackage{bm}
\usepackage{accents}
\usepackage{xcolor}
\usepackage{cite}

\usepackage{comment}

\newif\ifitsdraft
\def\itsdraft{\global\itsdrafttrue}

\itsdraft

\begin{document}

\title{
\textbf{On the Perturbed Projection-Based Distributed 
Gradient-Descent Algorithm: A Fully-Distributed Adaptive Redesign}
}
\author{Tarek Bazizi, Mohamed Maghenem, Paolo Frasca, Antonio Lor{\'i}a, and Elena Panteley 
\thanks{T. Bazizi and P. Frasca are with  Univ.\ Grenoble Alpes, CNRS, Inria, Grenoble INP, GIPSA-lab, 38000 Grenoble, France. 
M. Maghenem is with  Univ. Grenoble Alpes, CNRS, Grenoble INP, GIPSA-lab, 38000 Grenoble, France. A. Lor{\'i}a and E. Panteley are with L2S, CNRS, Gif-sur-Yvette, France. }
}

\maketitle

\begin{abstract}
In this work, we revisit a classical distributed gradient-descent algorithm, introducing an interesting class of perturbed multi-agent systems. The state of each subsystem represents a local estimate of a solution to the global optimization problem. 
Thereby, the network is required to minimize local cost functions, while gathering the local estimates around a common value. 
Such a complex task suggests the interplay of consensus-based dynamics with
gradient-descent dynamics. The latter descent dynamics involves the projection operator, which is assumed to provide corrupted projections of a specific form, reminiscent of existing (fast) projection algorithms. 
Hence, for the resulting class of perturbed networks, we are able to adaptively tune some gains in a fully distributed fashion, to approach the optimal consensus set up to arbitrary-desired precision. 
\end{abstract}

\section{INTRODUCTION}

In recent years, distributed optimization has gained a significant attention due to  various applications, including resource sharing/allocation \cite{doostmohammadian_survey_2025},  sensor networks \cite{Rabbat_sensor_networks}, and electric power systems \cite{Molzahn_power_sys}, among others.
Distributed optimization aims at minimizing a global cost function by distributing the computation across multiple agents (computation units). 
Each agent disposes of its own local cost function and the sum of the local cost functions represents the global cost function. The agents exploit only their own local cost function and interact with neighbor agents to coordinate and track a minimum value for the global cost function.

In the literature, various algorithms have been developed to address distributed optimization problems \cite{yang_survey_2019}. 
In the 
discrete-time setting, \cite{Nedic_2009} introduces a distributed version of the gradient-descent algorithm with a decreasing step-size, allowing the gradient influence to gradually vanish in time, which allows the network to achieve synchronization. This approach is extended to fixed step-sizes in \cite{EXTRA_shi_wei, Nedic_fixed_step_2017, YANG2018182_PI} to avoid harming the convergence to the optimal solution. Additionally, \cite{Nedic_distributed_step_2017} and \cite{Aug-DGM} propose an algorithm, in which every agent admits a distinct but fixed step-size. Other methods based on a dual formulation, such as distributed ADMM, are also considered in \cite{distributedADMM_tuto}. 

In the continuous-time setting, using Lyapunov stability theory,  distributed proportional-integral  algorithms are considered in \cite{wang_elia, gharesifard_distributed_2014},  to mitigate the synchronization error. Furthermore, given that each agent knows its own local cost function only, \cite{shi_reaching_2013, wei_ren_rendezvous_empty, lou_approximate_2014, lou_distributed_2016} assume that each agent can reconstruct the set of minima of its local cost function (the local optimal set). Thereby, transforming the local optimization problems into minimizing the distance with respect to  the local optimal set. This approach removes the need to have differentiable local cost functions, and requires the local optimal sets to be convex. 
In \cite{lou_approximate_2014} and \cite{lou_distributed_2016}, the projections to the local sets are corrupted by numerical perturbations. The first approach assumes that the erroneous projection belongs to a restricted triangular area including the actual projection. The second approach assumes that the perturbed projection lies within the boundary of the local optimal set. Both approaches characterize these errors using the angle created by the actual and the erroneous projections, assuming this error to decrease over time.

In this work, we revisit the class of distributed gradient-descent algorithms studied in the aforementioned literature. That is, we consider a multi-agent system, in which each subsystem dynamically updates its state to minimize a local cost function, while ensuring that the network collectively approaches a common minimizer. This complex task involves the interplay of a consensus-based dynamics, governed by a strongly-connected directed graph, and a gradient-descent dynamics. The latter descent dynamics involves a projection operator, which is assumed to provide perturbed projections as typical in (fast) projection algorithms \cite{usmanova2021fast}. Thanks to our approach, we are able to adaptively tune some consensus  and descent gains, so that the solutions of the resulting class of perturbed networks approach the optimal consensus set up to arbitrary-desired precision. Different from the existing works 
\cite{lou_approximate_2014,lou_distributed_2016}, the designed gains are dynamically updated and are not predefined  off-line as time-varying signals. Furthermore, their design, inspired by \cite{chen_adaptive_2024}, is fully distributed, in the sense that it requires solely the knowledge of desired precision. Neither the graph properties, size, nor the magnitude of the perturbation are required to be known.   

The rest of this paper is organized as follows. Preliminaries are in Section~\ref{Sec.Prelim}. The problem formulation is in Section~\ref{Sec.Probform}. The main result is in Section~\ref{Sec.MainRes}. Finally, numerical simulations are in Section~\ref{Sec.Sim}. 
\ifitsdraft
\else 
Due to space constraints, some proofs and intermediate results are omitted 
and can be found in \cite{Appendix--}.
\fi

\textbf{Notation.}
For $v  \in \mathbb{R}^N$,  $v^\top$ denotes its transpose, $|v|$ denotes its Euclidean norm, and $\text{diag}\{v\} = \text{diag}\{v_1,v_2,...,v_N\} \in \mathbb{R}^{N \times N}$ denotes the diagonal matrix whose $i$-th diagonal element is $v_i$, the $i$-th element of $v$. For a set $K \subset \mathbb{R}^N$,  $|v|_K := \min \{|x-y| : y \in K \}$ denotes the distance of $x$ to the set $K$. For a symmetric positive semi-definite matrix 
$Q \in \mathbb{R}^{N \times N}$,  $\lambda_i(Q)$ denotes the $i$th smallest eigenvalue of $Q$. 
We use $\mathbb{D}^N$ to denote the set of diagonal matrices of dimension $N$, and we let 
$\mathbb{B}^N := \{ D \in \mathbb{D}^N : |D_{ii}| \leq 1~~ \forall i \in 
\{1,2,..,N\} \}$. A class $\mathcal{K}$ function  $\alpha : \mathbb{R}_{\geq 0} \rightarrow \mathbb{R}_{\geq 0}$ is continuous, increasing, and $\alpha(0) = 0$.  A class $\mathcal{K}_\infty$ function  $\alpha : \mathbb{R}_{\geq 0} \rightarrow \mathbb{R}_{\geq 0}$ satisfies $\alpha \in \mathcal K$ and $\alpha(s) \to \infty$ as $s \to \infty$.
A class $\mathcal L$ function $\sigma : \mathbb{R}_{\geq 0} \rightarrow \mathbb{R}_{\geq 0}$ is continuous, non-increasing, and $\sigma(s) \to \infty$ as $s \to \infty$. A class $\mathcal K \mathcal L$ function $\beta : \mathbb{R}_{\geq 0} \times \mathbb{R}_{\geq 0} \rightarrow \mathbb{R}_{\geq 0}$ satisfies $\beta(.,t) \in \mathcal K$ for any fixed $t\geq 0$ and $\beta(s,.) \in \mathcal L$ for any fixed $s \geq 0$. By $ \bold 1_N := [1~1~...~1]^\top$, we denote the vector of dimension $N$ whose entries are equal to $1$. For a continuously differentiable function $g : \mathbb{R}^N \rightarrow \mathbb{R}$, $\nabla g :  \mathbb{R}^N \rightarrow \mathbb{R}^N$ denotes the gradient of $g$. Finally, we use $\mathbb{B}$ to denote the closed unit ball centered at the origin. 

\section{Preliminaries} \label{Sec.Prelim}
\subsection{The projection map}
Consider a closed subset $K \subset \mathbb{R}^n$ and let $x \in \mathbb{R}^n$. 
The projection of $x$ on $K$ is given by
\begin{align} \label{eqproj} 
\Pi_K(x) := \text{argmin} \{|x-y| : y \in K \}. 
\end{align}
When $K$ is convex, it is shown in \cite[Page 23-24]{aubin_differential_1984} that
\begin{align} \label{eqvarproj}
|\Pi_K(x)-\Pi_K(y)| \leq |x-y| \qquad  \forall x, y \in \mathbb{R}^n.
\end{align}
Moreover, under the convexity of the set $K$, the map $\Pi_K(x)$ is a singleton, $x \mapsto |x|^2_K := |x - \Pi_K(x)|^2$ is continuously differentiable, and 
 \begin{align} \label{eqCDdist}
 \nabla|x|^2_K = 2(x-\Pi_K(x)) \qquad \forall x \in \mathbb{R}^n. 
 \end{align}
Additionally, given $x$, $y \in \mathbb{R}^n$ and 
$A \subseteq B \subset \mathbb{R}^n$ closed and convex, it holds that \cite{shi_reaching_2013}
\begin{align}   
\label{proj_ineq_1}
    (x-\Pi_{B}(x))^\top (\Pi_{A}(x) - \Pi_{B}(x)) & \leq  0, 
    \\
  -(x-\Pi_{B}(x))^\top(x-\Pi_{A}(x))   &\leq - |x|_{B}^2.  \label{proj_ineq_1+}
\end{align} 

\subsection{Graph theory}
A directed graph, or di-graph,  $\mathcal{G}(\mathcal{V},\mathcal{E})$ is characterized by a set of nodes $\mathcal{V} := \{ 1,2,...,N \}$ and a set of directed edges $\mathcal{E}$.  The set of edges $\mathcal{E}$ consists of ordered pairs of the form 
$(k,i)$,  which indicates a directed link from node $k$ to node $i$. We consider directed graphs with no self-arcs. Note that when there exists a directed edge $(k,i) \in \mathcal{E}$,  then node $k$ is called an \emph{in-neighbor} of node $i$.  
We use  $\mathcal{N}_i \subset 
\{ 1,2,...,N \} \backslash \{ i \}$ to denote the indexes of the in-neighbors of agent $i$.  Furthermore,  we assign a weight $a_{ik} \geq 0$ to each edge $(k,i)$.  That is,  $a_{ik} = 0$ if $(k,i) \notin \mathcal{E}$.  
The Laplacian matrix is denoted by
$L := D - A$,
where $D := \text{diag} \{ d_1,d_2,...,d_N \}$ with
$d_i := \sum^N_{j=1}  a_{ij}$ for all $i \in 
\{ 1,2,...,N \}$ is the diagonal part of $L$,  and $A := [a_{ij}]$ is the \textit{adjacency} matrix of $\mathcal{G}$.  The directed graph $\mathcal{G}$ is \textit{strongly connected} if,  for any two distinct nodes $i$ and $j$,  there is a path from $i$ to $j$.  

\section{ Problem formulation} \label{Sec.Probform}

Consider the unconstrained optimization problem 
\begin{equation}
\label{OptUnc}
\min_{y \in \mathbb{R}^n} \sum_{i=1}^{N} f_i(y),   
\end{equation}
where each $f_i:\mathbb{R}^n \rightarrow \mathbb{R}$ is assumed to be convex. 

In  distributed optimization approaches,  different computation units are used to collaboratively solve \eqref{OptUnc}.  In particular,   every $f_i$ is assumed to be known to a single unit only.  Hence, a key step consists in recasting \eqref{OptUnc} into the  
constraint-coupled 
optimization problem 
\begin{equation}
\label{optimization problem (fi)}
\begin{split}
    \min_{x \in \mathbb{R}^{nN}} &\sum_{i=1}^{N} f_i(x_i), \quad x := [x_1^\top ~ x_2^\top ~...~ x_N^\top], \\
    \textrm{s.t.} \quad  &x_i=x_j \quad \textrm{for all} \quad i,j \in \mathcal{V}. 
\end{split}    
\end{equation}
Furthermore, for every $i \in \{1,2,...,N\}$, we use  $X_i \subset \mathbb{R}^n$ to denote the set of points where the cost function $f_i$ reaches its minimum. Note that every $X_i$ is convex since so is the corresponding function $f_i$. 
Now, by assuming that   
\begin{align} \label{eqinters}
X_o := \bigcap_{i=1}^{N} X_i \neq \emptyset,
\end{align}  
we can further recast \eqref{OptUnc} into the quadratic-cost optimization problem 
\begin{equation}
\label{optimization problem (projections)}
\begin{aligned}
\min_{x \in \mathbb{R}^{nN}} &  \frac{1}{2}\sum_{i=1}^{N} |x_i|_{X_i}^2 \\
   \textrm{s.t.} \quad &  x_i=x_j \quad \textrm{for all} \quad i,j \in \mathcal{V}. 
   \end{aligned}
\end{equation}

\subsection{A Class of Networked Systems}

According to \cite{shi_reaching_2013}, \cite{lou_approximate_2014}, and \cite{lou_distributed_2016}, we can solve \eqref{OptUnc} by solving 
\eqref{optimization problem (projections)} dynamically 
through a distributed gradient-descent algorithm. 
More specifically,  we can design a network of multi-agent systems,  where the state of each subsystem,  denoted by $x_i$,  represents a local estimate of a solution to~\eqref{OptUnc}. The design of this network must achieve two concurrent goals: 
1 - Every agent minimizes the corresponding local cost function. 2- The local minimizers, generated by the agents, converge to a common value. 

To achieve the first goal, a
gradient-descent dynamics minimizing the quadratic cost function 
$\frac{1}{2}|\cdot|^2_{X_i}$ is added,  thanks to 
the continuous differentiability of  $\frac{1}{2}|\cdot|^2_{X_i}$; see \eqref{eqCDdist}. 
To achieve the second goal, an interconnection protocol, governed by the directed graph $\mathcal{G}(\mathcal{V}, \mathcal{E})$, is needed. 
The following graph assumption is made. 
\begin{assumption} \label{assGPH}
The graph $\mathcal{G}$ is strongly connected.  
\end{assumption}

In summary, every agent $i \in \mathcal{V}$ updates dynamically its local solution to \eqref{OptUnc}, denoted by $x_i$, according to 
\begin{equation} \label{system_xi}
\dot{x}_i = \gamma_i \alpha_i \sum_{j = 1}^{N} a_{ij} (x_j -x_i) - \alpha_i (x_i - \Tilde{\Pi}_{X_i}(x_i)),
\end{equation}
where, for each $i \in \mathcal{V}$, $\gamma_i > 0$ and $\alpha_i > 1$ are, respectively, the coupling and the descent gains that we will adaptively design later.  
Furthermore, for each $i \in \mathcal{V}$,  the map $\Tilde{\Pi}_{X_i}: \mathbb{R}^n \rightarrow  \mathbb{R}^n$ represents a corrupted (numerical) estimate of the projection map $\Pi_{X_i}$ introduced in \eqref{eqproj}, and is given by  
\begin{equation}
\label{corrupted projection}
\Tilde{\Pi}_{X_i}(x_i) := \Pi_{X_i}(x_i) 
+ \zeta_i p_i(x_i),
\end{equation}
where $\zeta_i \in (0,1)$ indicates the precision of the numerical projection, and $p_i:\mathbb{R}^n \rightarrow \mathbb{R}^n$ stands for the worst-case projection error, verifying the following assumption. 
\begin{assumption} \label{Assbd}
There exists $\Bar{p}>0$ such that 
$$ \sup_{y \in \mathbb{R}^n} |p_i(y)| \leq  \Bar{p} \qquad \forall i \in 
\mathcal{V}. $$
\end{assumption}

\begin{remark}
Approximate projections verifying \eqref{corrupted projection} and Assumption~\ref{Assbd} are consistent with various iterative (fast) projection algorithms; see \cite{usmanova2021fast} and \cite{JMLR:v18:16-147}. In those algorithms,  higher   precisions (i.e., small values of $\zeta_i$) are achieved at the price of larger number of iterations. Furthermore, the worst-case projection $p_i$ is usually unknown but bounded and can result from solving dual problems \cite{usmanova2021fast}.  
\end{remark}

Our framework is based on the key assumption that we are able to tune every projection precision $\zeta_i$  to guarantee 
$$ \zeta_i \leq  1/\alpha_i \qquad \forall i \in \mathcal{V}, $$
where $\alpha_i$ is the gradient-descent gain introduced in \eqref{system_xi}.
As a consequence, we recover the class of perturbed multi-agent systems of the form 
\begin{align}
\dot{x}_i = \gamma_i \alpha_i \sum_{j = 1}^{N} a_{ij} (x_j -x_i)  & - \alpha_i (x_i - \Pi_{X_i}(x_i))  + p_i(x_i)
   \nonumber  \\ &
  \forall i  \in 
\mathcal{V}.  \label{system_xinew}
    \end{align}

\begin{remark} \label{remElena}
In the particular setting where 
$$\gamma := \gamma_1 \alpha_1 = ... = \gamma_N \alpha_N > 0, $$
the networked system in \eqref{system_xinew} belongs to the class of networks studied in \cite{lazri:hal-04609468}, where the term $- \alpha_i (x_i - \Pi_{X_i}(x_i))$ is replaced by a general function $f_i(x_i)$ such that $u_i \mapsto x_i$, under $\dot{x}_i = f_i(x_i) + u_i$, 
defines a semi-passive map. 
See also \cite{10298266,7809144} for the unperturbed case. The goal in the aforementioned 
works is restrained to guaranteeing boundedness and ultimate boundedness of the network's solutions, as well as practical synchronization, i.e., showing that, for every $\varepsilon>0$, there exists   $\gamma^* > 0$ such that, for each $\gamma \geq \gamma^*$,  
$$ \lim_{t \rightarrow + \infty} |x_i(t) - x_j(t)| \leq \varepsilon \qquad \forall i,j \in \mathcal{V}. $$  
In fact, adapting the 
Lyapunov-based approach in \cite{lazri:hal-04609468, 10298266, 7809144} to our setting is key to asymptotically fulfill the constraint in \eqref{optimization problem (projections)} up to a desired precision $\varepsilon$.    
\end{remark}

\subsection{Adaptive design}

A dynamical way to solve \eqref{OptUnc} is by guaranteeing  asymptotic convergence towards the optimal consensus set 
$$ \mathcal{A}:= \{x\in X_1 \times ... \times X_N : x_1 = \dots = x_N\} $$ 
for the multi-agent system \eqref{system_xinew}. 
In the unperturbed case, when $p \equiv 0$, this problem is solved in \cite{shi_reaching_2013} for any $\alpha_1, \gamma_1, ..., \alpha_N, \gamma_N > 0$. 
However, it is out of reach in the presence of the perturbation $p$. 
Instead, through a fully-distributed adaptive design of the sequence 
$\{(\gamma_i,\alpha_i)\}^{N}_{i=1}$, we establish in this work convergence of the solutions to an arbitrarily-small neighborhood of the set $\mathcal{A}$. 
More specifically, we address the following problem. 
\begin{problem} \label{PB1}
Given $\varepsilon > 0$,
we design 
$\{(\gamma_i, \alpha_i)\}^{N}_{i=1}$ as
\begin{align} 
\dot{\gamma}_i & = 
h_i(x_i,\gamma_i,\{(x_j, \gamma_j)\}_{j \in \mathcal{N}_i}, \varepsilon), 
\label{eqdesgamma}
\\
\dot{\alpha}_i & = 
g_i(x_i,\alpha_i, \{(x_j,\alpha_j)\}_{j \in \mathcal{N}_i}, \varepsilon), \label{eqdesalpha}
\end{align}
so that the multi-agent system \eqref{system_xinew} under  \eqref{eqdesgamma}-\eqref{eqdesalpha} achieves 
\begin{align} \label{eqObject1} 
\lim_{t \rightarrow +\infty} 
|x(t)|_{\mathcal{A}_s} & \leq \varepsilon, ~~ {\mathcal{A}_s} := \{x_i = x_j ~~ \forall i,j \in \mathcal{V} \},
\\    
\lim_{t \rightarrow +\infty} 
|x(t)|_{\mathcal{A}_o} & \leq \varepsilon, ~~ \mathcal{A}_o := X_1 \times ... \times X_N, 
\label{eqObject2} 
\end{align}
and, at the same time,
\begin{align} \label{eqObject+} 
\max \{ | \gamma_i |_\infty, | \alpha_i |_\infty  \}  < \infty \qquad \forall i \in \mathcal{V}. 
\end{align}
\end{problem}

\begin{remark}
It is important to note that verifying  \eqref{eqObject1}-\eqref{eqObject2} does not necessarily 
 lead to 
$$ \lim_{t \rightarrow +\infty} 
|x(t)|_{\mathcal{A}} \leq \varepsilon, $$
even though $\mathcal{A} = \mathcal{A}_s \cap \mathcal{A}_o$. However, 
since both functions 
$$x \mapsto |x|^2_{\mathcal{A}} ~~ \text{and} ~~ x \mapsto 
\max \{ |x|^2_{\mathcal{A}_o}, |x|^2_{\mathcal{A}_s} \} $$ 
are continuous, positive definite with respect to the compact set $\mathcal{A}$, and radially unbounded, we conclude, according to \cite{kellett2014compendium},  the existence of $\rho \in \mathcal{K}$ such that
 $$|x|_{\mathcal{A}} \leq \rho(\max \{ |x|_{\mathcal{A}_s}, |x|_{\mathcal{A}_o} \}) \quad \forall x \in \mathbb{R}^{nN}. $$  
Furthermore,  the function $\rho$ would be linear, according to \cite[Theorem 1]{19921570108}, provided that there exists $\delta > 0$ such that 
$$ \max_{y \in \mathcal{A}} |y|_{\mathbb{R}^{nN} \backslash \mathcal{A}_o} > \delta. $$
In particular, by solving Problem~\ref{PB1}, we only guarantee that 
$$ \lim_{t \rightarrow +\infty} 
|x(t)|_{\mathcal{A}} \leq \rho(\varepsilon), $$
Hence, according to our framework, guaranteeing convergence of $x$ to an arbitrarily-small neighborhood of $\mathcal{A}$ requires the knowledge of the function $\rho$.   
\end{remark}

Different from the existing literature \cite{lou_approximate_2014,lou_distributed_2016}, the design of $\{(\gamma_i, \alpha_i)\}^{N}_{i=1}$ required in Problem~\ref{PB1} is adaptive and does not rely on predefined time-varying exogenous signals. Furthermore, the design we propose is fully distributed, in the sense that it requires solely the knowledge of the precision $\varepsilon$. Neither the graph spectrum, size, nor the magnitude of the perturbation $p$ are needed to be known.  

\section{Main Result} \label{Sec.MainRes}

We propose to dynamically update the sequence 
$\{(\gamma_i, \alpha_i)\}^{N}_{i=1}$ according to 
\begin{align}
    \dot{\gamma}_i & = k_\gamma \sum_{j \in \mathcal{N}_i}  a_{ij} (\gamma_j -\gamma_i) 
    \nonumber \\ &
    + \frac{1}{2} \text{dsat}_{\varepsilon} \left( |x_i|_{X_i}  \right) + 
    \frac{1}{2}
 \text{dsat}_{\varepsilon} \left(\sum_{j \in \mathcal{N}_i}  |x_j -x_i| \right), \label{xi-gamma}
    \\
  \dot{\alpha}_i & =   k_\alpha \sum_{j \in \mathcal{N}_i} a_{ij} (\alpha_j -\alpha_i) \nonumber \\ & + \frac{1}{2} \text{dsat}_{\varepsilon} \left( |x_i|_{X_i}  \right) + \frac{1}{2} \text{dsat}_\varepsilon \left( \sum_{j \in \mathcal{N}_i}  |x_j -x_i| \right),
  \label{xi-alpha}
\end{align}
where $k_\alpha$, $k_\gamma > 0$ are free design parameters,  and the function $\text{dsat}_\varepsilon : \mathbb{R}_{\geq 0} \rightarrow \mathbb{R}_{\geq 0}$ is given by
\begin{equation}
\text{dsat}_\varepsilon(a) := \frac{1+|a-\varepsilon|-|a-\varepsilon-1|}{2} \quad \forall a \geq 0.
\end{equation}

\begin{remark}
The function $\text{dsat}_\varepsilon$ is non-negative and bounded from above by $1$, vanishes on $[0,\varepsilon]$, and linear on $[\varepsilon, \varepsilon + 1]$. 
\end{remark}

The adaptive design in 
\eqref{xi-gamma}-\eqref{xi-alpha}
is inspired by \cite{chen_adaptive_2024}, where a synchronization problem for heterogeneous multi-agent systems is addressed. Roughly speaking, 
using boundedness of  $\text{dsat}_\varepsilon$, we can show that the $\gamma_i$s (resp., the $\alpha_i$s), after some time, must remain close to their convex combination $\gamma_s := \sum^{N}_{i = 1} v_i \gamma_i$ (resp., $\alpha_s := \sum^{N}_{i = 1} v_i \alpha_i$), for some $\{ v_i \}^N_{i=1} \subset (0,1)$ with  $\sum^N_{i=1} v_i = 1$.   
Hence, for all $i \in \mathcal{V}$, we can write 
$\gamma_i = \gamma_s + \tilde{\gamma}_i$ (resp., $\alpha_i = \alpha_s + \tilde{\alpha}_i$), for a bounded residual $\tilde{\gamma}_i$ (resp.,  $\tilde{\alpha}_i$).  
As a result, by accommodating the study in \cite{lazri:hal-04609468} in the presence of $\tilde{\gamma}_i$ (resp., $\tilde{\alpha}_i$), see Remark~\ref{remElena}, we guarantee convergence of $t \mapsto x(t)$ to the set $\mathcal{A}_s + \varepsilon \mathbb{B}$ provided that $\gamma_s$ and $\alpha_s$ exceed specific values. 
At the same time, the analysis of $\dot{\gamma}_s$ and $\dot{\alpha}_s$ allows us to conclude that a necessary condition for $\gamma_s$ and $\alpha_s$ to stop increasing is that the solution $t \mapsto x(t)$ reaches the set $\mathcal{A}_s + \varepsilon \mathbb{B}$. Thereby, we must 
fulfill \eqref{eqObject1} at some point. 
We might still witness the increase of $\alpha_s$ and $\gamma_s$ until  
convergence of $t \mapsto x(t)$ to $\mathcal{A}_o + \varepsilon \mathbb{B}$
is additionally achieved. Indeed, we show, via leveraging some Lyapunov-based tools, that the latter convergence property must occur for finite values of $\alpha_s$ and $\gamma_s$. Thereby, we guarantee both boundedness of $\{(\alpha_i, \gamma_i)\}^N_{i=1}$ and \eqref{eqObject2}.

We are now ready to state our main result. 

\begin{theorem} \label{thm1}
Under Assumptions~\ref{assGPH}-\ref{Assbd}, the multi-agent system in \eqref{system_xinew} under \eqref{xi-gamma}-\eqref{xi-alpha} solves Problem~\ref{PB1}. 
\end{theorem}

\begin{proof}
We start introducing the vector $v \in \mathbb{R}^N$ such that $v^\top L = 0$ and $v^\top  \bold 1_N = 1$. Under Assumption~\ref{assGPH}, we conclude that such a  vector $v$ is unique and that its entries are positive; see \cite[Lemma 1]{lazri:hal-04609468}.

The proof of Theorem~\ref{thm1} is based on two steps. 

\subsubsection*{Step 1 
(Verifying \eqref{eqObject1})}
We do so using the following four  intermediate results. 

\begin{itemize}
\item[(I)] We show that the design in \eqref{xi-gamma}-\eqref{xi-alpha} forces the matrices 
$$ \Gamma := \text{diag}\{\gamma_1,...,\gamma_N\} \text{ and } \Theta := \text{diag}\{ \alpha_1,...,\alpha_N \}
$$
to reach and remain within the set 
\begin{align*} 
\hspace{-0.3cm} \mathcal{D} := \left\{ D \in \mathbb{D}^N: D \in d_s I_N + d_e \mathbb{B}^N, ~ d_s := \bold{1}_{N}^\top D v  \right\}, 
\end{align*}
for some $d_e > 0$ and after some time $T_e \geq 0$. 
\ifitsdraft
See Lemma~\ref{lem-gamma-ball}.  
\else
See \cite[Lemma 1]{Appendix--}.
\fi

\item[(II)]
We show global uniform (in $\Gamma$ and $\Theta$) boundedness (GUB) of trajectories. That is,  
we show the existence of $\gamma_o > 0$ and $\alpha_o > 1$ such that, for every $r_{o} > 0$, there exists $R := R(r_{o}, \gamma_{o},\alpha_o) \geq r_{o}$ such that,  for all 
$\Gamma, \Theta : \mathbb{D}^N \rightarrow \mathcal{D}$ with 
\begin{equation}
\label{eq_s} 
\begin{aligned}  
 \gamma_s(t) & := \left(\bold{1}_{N}^\top \Gamma(t) v \right) \geq \gamma_o \qquad \forall t \geq 0, 
 \\ 
 \alpha_s(t) & := \left(\bold{1}_{N}^\top \Theta(t) v \right) \geq \alpha_o \qquad \forall t \geq 0, 
\end{aligned}
\end{equation}
we have 
\[
|x(0)| \leq r_{o} \Longrightarrow |x(t)| \leq R  \qquad \forall t  \geq 0.
\]
\ifitsdraft
See Lemma~\ref{lemGUUB}.
\else 
See \cite[Lemma 3]{Appendix--}.
\fi 

\item[(III)] We show global uniform (in $\Gamma$ and $\Theta$) ultimate boundness (GUUB). That is,  we guarantee the existence of $\gamma_{o} > 0$, $\alpha_o > 1$, and $r =  r(\gamma_o, \alpha_o) > 0$ such that, for all $r_{o} >0$, there exists $T = T(r_{o},\gamma_{o}, \alpha_o) \geq 0$ such that,  for all 
$\Gamma, \Theta: \mathbb{D}^N \rightarrow \mathcal{D}$ with 
$ \gamma_s(t) \geq \gamma_o$ and $\alpha_s(t) \geq \alpha_o$ for all $t \geq 0$, 
we have 
\begin{align*}
  |x(0)| \leq r_{o} 
\Longrightarrow |x(t)| \leq r \quad \forall t \geq  T. 
\end{align*}
\ifitsdraft
See Lemma~\ref{lemGUUB}.
\else 
See \cite[Lemma 3]{Appendix--}.
\fi

\item[(IV)] 
We show global practical synchronization. That is, 
we guarantee the existence $\beta_1$, $\beta_2 \in \mathcal{K L}$ and $\alpha^* > \alpha_o > 1$ such that, for any $\varepsilon > 0$, there exists $\gamma^*(\varepsilon) > \gamma_o > 0$ such
that, for all $\Gamma, \Theta: \mathbb{D}^N \rightarrow \mathcal{D}$ with 
$\gamma_s(t) \geq \gamma^*$ and $\alpha_s(t) \geq \alpha^*$ for all $t \geq 0$, we have  
\begin{align} \label{eqpropty1}
\hspace{-0.6cm} |x(t)|_{\mathcal{A}_s}^2 \leq  \frac{\varepsilon}{2} + \beta_1(|x(0)|_{\mathcal{A}_s}, \alpha^* t) + \frac{\beta_2(|x(0)|, \alpha^* t)}{\gamma^*}
\end{align}  
for each $(x(0),t) \in \mathbb{R}^{nN} \times \mathbb{R}_{\geq 0}$. 
\ifitsdraft
See Lemma~\ref{lemGPracS}.
\else
See \cite[Lemma 4]{Appendix--}.
\fi

\end{itemize}
To complete the proof of \eqref{eqObject1}, we reason on the time interval $[T_e, + \infty)$. Indeed, 
one can easily show that the solutions cannot blow up on the interval $[0,T_e]$. 
Next, using property (I), we conclude that 
$$(\Gamma(t), \Theta(t)) \in \mathcal{D} \times \mathcal{D} \qquad \forall t \geq T_e. $$
Furthermore, we note that 
\begin{align*} 
& \dot{\alpha}_s = \dot{\gamma}_s  =
\\ &
\frac{1}{2} \sum^{N}_{i=1} v_i \left[ \text{dsat}_\varepsilon \left(\sum_{j \in \mathcal{N}_i}  |x_j -x_i| \right) + \text{dsat}_\varepsilon \left( |x_i|_{X_i} \right)  \right] \geq 0.
\end{align*}
Hence, the variables 
$\gamma_s$ and 
$\alpha_s$ are doomed to increase if the solution $t \mapsto x(t)$ remains away from $\left( \mathcal{A}_s + \varepsilon \mathbb{B} \right) \cap \left( \mathcal{A}_o + \varepsilon \mathbb{B} \right)$.
However, using property (IV),  we conclude that once 
$\gamma_s$ passes the value $\gamma^*$ and $\alpha_s$ passes $\alpha^*$, the solution $t \mapsto x(t)$ must enter and remain within 
$\mathcal{A}_s + \varepsilon \mathbb{B}$,
leading to 
\begin{align}  \label{eqgover}
\dot{\alpha}_s = \dot{\gamma}_s  =
\frac{1}{2} \sum^{N}_{i=1} v_i  \text{dsat}_\varepsilon \left( |x_i|_{X_i} \right)  \geq 0.
\end{align} 

\subsubsection*{Step 2 
(Verifying \eqref{eqObject2} and \eqref{eqObject+})}

To do so, we use  
the following property, established in \ifitsdraft Lemma~\ref{lemadded}.
\else
\cite[Lemma 5]{Appendix--}.
\fi

\begin{itemize}
\item[(V)] We show global practical optimality. That is, given $\varepsilon > 0$,  we guarantee 
the existence of $\alpha^{**}(\varepsilon) \geq \alpha^*$ and $\gamma^{**}(\varepsilon) \geq \gamma^{*}(\varepsilon)$ such that, for all $r_o > 0$, there exists $T_o = T_o(r_o, \alpha^{**}, \gamma^{**}) \geq 0$ such that, for all $\Gamma, \Theta : \mathbb{D}^N \rightarrow  \mathcal{D}$ with $\gamma_s(t) \geq \gamma^{**}(\varepsilon)$ and $\alpha_s(t) \geq \alpha^{**}(\varepsilon)$ for all $t \geq 0$ and for all solution $t \mapsto x(t)$ with $|x(0)| \leq r_o$, we have 
\begin{align} \label{eqopti} 
|x_i(t)|_{X_i} \leq \varepsilon \qquad \forall t \geq T_o, ~~~ \forall i \in \mathcal{V}. 
\end{align}
\end{itemize}

As in Step 1, we reason on the time interval $[T_e, + \infty)$ since, after property (I), we have 
$$(\Gamma(t), \Theta(t)) \in \mathcal{D} \times \mathcal{D} \qquad \forall t \geq T_e. $$
Furthermore, after Step 1, we conclude the existence of $T'_e \geq T_e$, after which, $\alpha_s$ and $\gamma_s$  are governed by \eqref{eqgover}. Hence, either $\alpha_s$ and $\gamma_s$ remain bounded implying that 
$ x(t) \rightarrow \mathcal{A}_o + \varepsilon \mathbb{B}$,
otherwise, there exists $T''_e \geq T'_e$, after which, $\alpha_s$ and $\gamma_s$ must be larger than $\alpha^{**}$ and $\gamma^{**}$, respectively. Hence, using property (V), we conclude that 
\begin{align*}  
|x_i(t)|_{X_i} \leq \varepsilon \quad \forall t \geq T''_e + T_o(|x(T''_e)|,\alpha^{**}, \gamma^{**}), ~ \forall i \in \mathcal{V}.
\end{align*}
This verifies \eqref{eqObject2} and 
lead to 
$$\dot{\alpha}_s(t) = \dot{\gamma}_s(t) = 0 \quad  \forall t \geq T''_e + T_o(|x(T''_e)|,\alpha^{**}, \gamma^{**}), $$ which guarantees boundedness of $(\alpha_s,\gamma_s)$.
\end{proof}

\section{Simulations} \label{Sec.Sim}
Consider the optimization problem in \eqref{optimization problem (fi)} with $N=3$ and  $n=2$. For each 
$i \in \{1,2,3\}$, the set of minimizers of the local cost function $f_i$ is given by 
\begin{align}
\nonumber
    &X_i := \left\{y \in \mathbb{R}^n ~:~ (y-c_i)^\top Q_q (y-c_i) \leq 1, q \in \{1,2,3\}\right\},
\end{align} 
where 
$$ Q_1 := \begin{bmatrix}
    1/3 & -1/3
    \\
    -1/3 & 7/9
\end{bmatrix}, 
Q_2 := \begin{bmatrix}
   1/3 & 1/3\\
    1/3 & 7/9
\end{bmatrix},  Q_3 := \frac{7}{10} I_2. $$
Furthermore,  $c_1 := (0, 0)$,
$c_2 := (0, 1)$, and $c_3 := (1, 0)$. 

The Laplacian matrix of 
the considered communication graph is given by 
\begin{equation*}
    L := \begin{bmatrix}
        1 & -1 & 0\\
        0 & 2 & -2\\
        -1 & -\frac{1}{2} & \frac{3}{2}
    \end{bmatrix}.
\end{equation*}
Given $i \in \{1,2,3\}$, the projection map $\tilde{\Pi}_{X_i}$ is computed via solving the  constrained optimization problem  
\begin{align*} 
  \text{min} |x_i - w|^2  \text{ s.t. }  w \in X_i
\end{align*}
 using the algorithm in \cite{usmanova2021fast}. The latter algorithm  is shown to approximate the true projection to 
 arbitrarily-desired precision. Such precision is adjustable at the expense of increasing the computation time. In particular, given a precision 
 $\zeta_i > 0$, 
 the algorithm in \cite{usmanova2021fast} provides, in $O(\text{log}(1/\zeta_i))$ steps, an approximation $\tilde{\Pi}_{X_i}(x_i)$ that verifies
\begin{equation*}
    \left|x_i - \tilde{\Pi}_{X_i}(x_i) \right|^2 \leq \left|x_i - \Pi_{X_i}(x_i) \right|^2 + \zeta_i p_i(x_i).
\end{equation*}
For some $y \mapsto 
\{p_i(y)\}^{3}_{i=1}$ verifying Assumption~\ref{Assbd}. 

For the proposed adaptive design of $\{\gamma_1, \gamma_2, \gamma_3 \}$ and 
$\{\alpha_1, \alpha_2, \alpha_3 \}$ in \eqref{xi-gamma} and \eqref{xi-alpha},   we use the precision $\varepsilon = 0.3$, the adaptation gains $k_\gamma = 10$, $k_\alpha = 1$, and the projection precisions to verify 
\begin{align} \label{eqalphazeta1}  
\alpha_i \zeta_i \leq 1  \qquad \forall i \in \{1,2,3 \}.  
\end{align}

The resulting trajectories are simulated on the interval 
$[0,5]$ with discretization step $\Delta t := 10^{-2}$. The agents are initialized at $x_1(0) = \begin{bmatrix}
    3 & 4
\end{bmatrix}^\top$,  $x_2(0) = \begin{bmatrix}
    5 & 3/2
\end{bmatrix}^\top$, and  
$x_3(0) = \begin{bmatrix}
    -3 & 4
\end{bmatrix}^\top$, and the gains are initially set to 
$$\Gamma(0) := 
\begin{bmatrix}
    -3/25 & 0 & 0 \\
    0 & 2/3 & 0 \\
    0 & 0 & 1
\end{bmatrix}, ~ \Theta(0) :=  \begin{bmatrix}
    1/4 & 0 & 0\\
    0 & -3/2 & 0\\
    0 & 0 & 1/2
\end{bmatrix}.$$

From the phase portrait depicted in Figure~\ref{phaseportrait_plot}, 
the three agents converge towards a neighborhood of the intersection set 
$X_o := X_1 \cap X_2 \cap X_3$. Furthermore, as shown in Figure~\ref{gammaalpha_plot}, the weights $\gamma_1$, $\gamma_2$, $\gamma_3$, $\alpha_1$, $\alpha_2$, and $\alpha_3$ become positive after some adaptation time and converge to some constant values. This validates the conclusions of Theorem~\ref{thm1}. 

Due to the explicit form of the sets $X_1$, $X_2$, and $X_3$, we are able to analytically compute the exact projections and compare them with there approximations using the considered algorithm: the projection errors are shown to be non-null, see Figure~\ref{proj_error_alphazetal leq 1}, and to increase if we relax the projection precision $\zeta_1$, $\zeta_2$ and $\zeta_3$ while satisfying  
\begin{align} \label{eqalphazeta2} 
\alpha_i \zeta_i \leq 300  \qquad \forall i \in \{1,2,3 \}.
\end{align}
  Finally, we can see in Figure~\ref{fig:four} that, under \eqref{eqalphazeta2}, the conclusions of Theorem~\ref{thm1} remain valid, although the projection errors are larger than under \eqref{eqalphazeta1}. 
\begin{figure}
    \centering
    \includegraphics[width=0.8\linewidth]{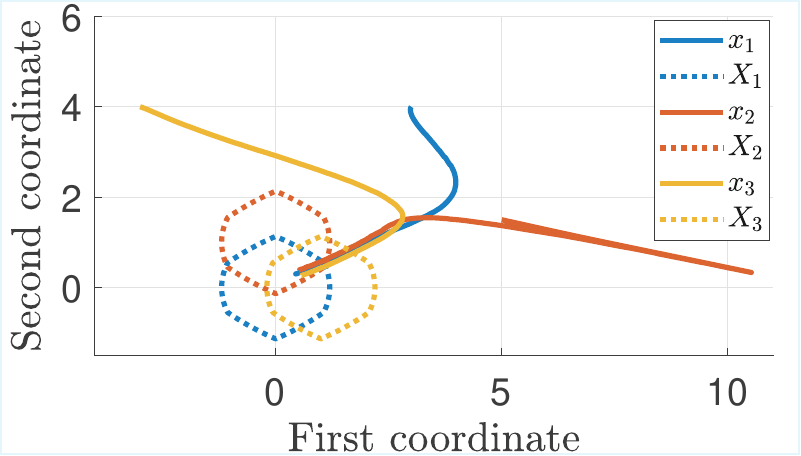}
    \caption{First two coordinates of the phase portrait of system \eqref{system_xinew} under \eqref{xi-gamma}-\eqref{xi-alpha} for the example system in Section~\ref{Sec.Sim}.}
    \label{phaseportrait_plot}
\end{figure}

\begin{figure}
    \centering
    \includegraphics[width=0.78\linewidth]{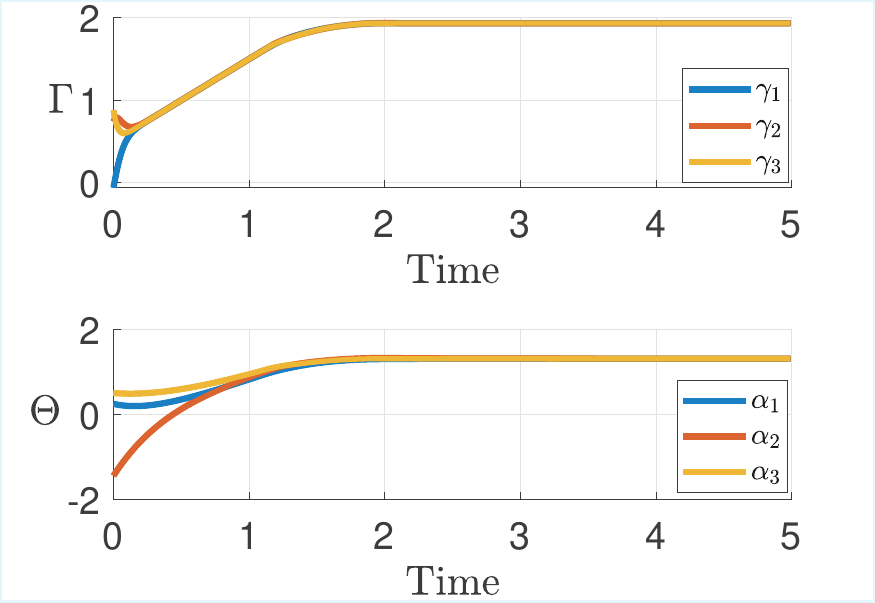}
    \caption{Evolution of $\Gamma$ and $\Theta$ according to \eqref{xi-gamma} and \eqref{xi-alpha}.}
    \label{gammaalpha_plot}
\end{figure}
\begin{figure}
    \centering
    \includegraphics[width=0.77\linewidth]{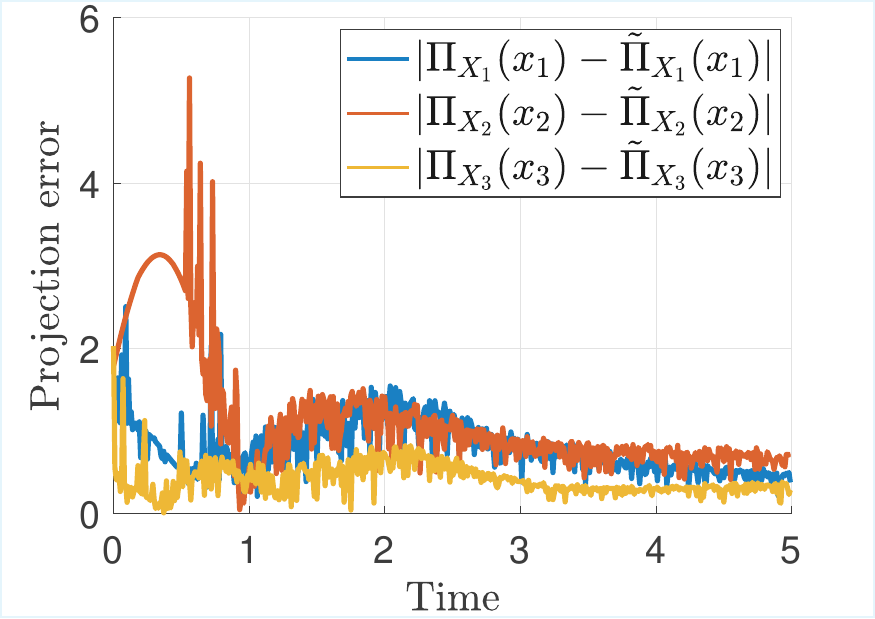}
    \caption{Projection errors with $\alpha_i\zeta_i \leq 1$.}
    \label{proj_error_alphazetal leq 1}
\end{figure}
\begin{figure}
    \centering
    \includegraphics[width=0.77\linewidth]{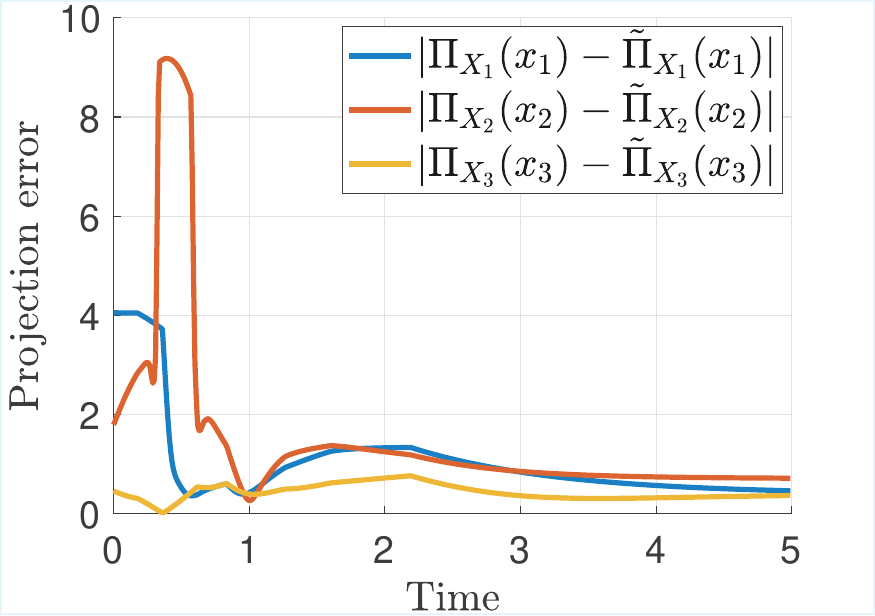}
    \caption{Projection errors with $\alpha_i\zeta_i \leq 300$.}
    \label{fig:four}
\end{figure}
\section{Conclusion}
This work proposed an adaptive and a fully-distributed redesign of an existing  distributed gradient-descent algorithm. The approach allowed us to handle a class of bounded perturbation that can be the result of using approximate (fast) projection algorithms. As a result, we are able to approach a solution to the global optimization problem up to arbitrary small precisions.  In the future, we aim to  relax the convexity of the local sets $(X_1,...,X_N)$ while using proximal methods, and assess the network's behavior when $\bigcap^{N}_{i=1} X_i = \emptyset$.  

\bibliographystyle{ieeetr}  
\bibliography{bib1}

\ifitsdraft

\appendix 

\begin{lemma} 
\label{lem-gamma-ball}
Consider the networked system \eqref{xi-gamma}-\eqref{xi-alpha} such that Assumption~\ref{assGPH} holds. Then, for $$ d_e :=  \frac{8 \max_{i}\{ v_i \}}{\min_{i}\{ v_i \} (k \lambda_2(Q))^2}, $$ 
we have 
$\Gamma(t)  \in \mathcal{D}$ and    
 $\Theta(t) \in \mathcal{D}$ 
 for all 
 $$t \geq T_\gamma := 2\frac{ \ln \left( \frac{d_e \min_i\{v_i\}}{2 W(\gamma(0))} \right) }{ k_\gamma \lambda_2(Q) } \max_i \{ v_i \} $$
and 
  for all 
 $$t \geq T_\alpha := 2 \frac{ \ln \left( \frac{d_e \min_i\{v_i\}}{2 W(\alpha(0))} \right) }{ k_\alpha \lambda_2(Q)  } \max_i \{ v_i \}, $$
 respectively, where $Q := L^\top \text{diag} \{v\} + \text{diag} 
 \{v\} L$ and $v := (v_{1},v_{2}, ...,  v_{N}) \in \text{Ker}(L^\top)$. 
\end{lemma}
\begin{proof}
Without loss of generality, we  prove the statement for $\Gamma$ only. 
We start letting 
$$  \gamma := \Gamma \bold 1_N \in \mathbb{R}^N   $$
and we introduce the Lyapunov function candidate 
\begin{align*}
    W(\gamma) & := \left(\gamma - \bold 1_N \gamma_s \right)^\top V \left(\gamma - \bold 1_N \gamma_s \right). 
\end{align*}
The derivative of $V$ along \eqref{xi-gamma} verifies 
\begin{align*}
  \dot{W}(\gamma) & =  - k_\gamma  \left(\gamma - \bold 1_N \gamma_s \right)^\top Q \left(\gamma - \bold 1_N \gamma_s \right) 
\\ &
+  2 \left(\gamma - \bold 1_N \gamma_s \right)^\top V g(x), 
\end{align*}
where $Q := L^\top V + V L $ and, for all $i \in \mathcal{V}$, we have 
$$  g_i(x) := \frac{1}{2}  \text{dsat}_\varepsilon \left( |x_i|_{X_i} \right) + 
\frac{1}{2}  \text{dsat}_\varepsilon \left(\sum_{j \in \mathcal{N}_i}  |x_j -x_i| \right).  $$
Next, using \cite[Lemma 1]{lazri:hal-04609468}, we conclude that $Q$ is positive semi-definite and that $\text{Ker}(Q) = \boldsymbol{1}_N$, which implies that 
$$  \left(\gamma - \bold 1_N \gamma_s \right)^\top Q \left(\gamma - \bold 1_N \gamma_s \right) \geq \lambda_2 (Q)  |\gamma - \bold 1_N \gamma_s|^2.  $$
On the other hand, we know that $ |g_i(x)| \leq 1$ for all $x \in \mathbb{R}^{nN}$ and for all $i \in \mathcal{V}$.  
As a result, we can write
\begin{align*}
  \dot{W}(\gamma) & \leq   - k_\gamma \lambda_2(Q) |\gamma - \bold 1_N \gamma_s |^2 
+  2 |\gamma - \bold 1_N \gamma_s| 
\\ & \leq 
 - \frac{k_\gamma \lambda_2(Q)}{2}  |\gamma - \bold 1_N \gamma_s |^2 
+  \frac{2}{k_\gamma \lambda_2(Q)}.  
\end{align*}
Next, we use the fact that 
$$ \min_{i}\{ v_i \} |\gamma - \bold 1_N \gamma_s |^2 \leq  W(\gamma) \leq \max_i \{ v_i \}  |\gamma - \bold 1_N \gamma_s |^2, $$
to obtain 
\begin{align*}
  \dot{W}(\gamma) &  \leq 
 - \frac{k_\gamma \lambda_2(Q)}{2 \max_{i}\{ v_i \}} W(\gamma) + \frac{2}{k_\gamma \lambda_2(Q)}.  
\end{align*}
Solving the latter inequality, we find
\begin{align*}
 W(\gamma(t)) & \leq W(\gamma(0))
 \exp^{ - \frac{k_\gamma \lambda_2(Q) t}{2 \max_{i}\{ v_i \}}  }  
 \\ & + \left[ 1 - \exp^{ - \frac{k_\gamma \lambda_2(Q) t}{2 \max_{i}\{ v_i \}} } \right] \frac{4 \max_{i}\{ v_i \}}{(k_\gamma \lambda_2(Q))^2},
\end{align*}
and thus 
\begin{align*}
 |\gamma - \bold 1_N \gamma_s |^2  & \leq \frac{W(\gamma(0))}{\min_{i}\{ v_i \}}
\exp^{ - \frac{k_\gamma \lambda_2(Q) t}{2 \max_{i}\{ v_i \}}  }  
\\ & + \left[ 1 - \exp^{ - \frac{k_\gamma \lambda_2(Q) t}{2 \max_{i}\{ v_i \}}  } \right] \frac{4 \max_{i}\{ v_i \}}{\min_{i}\{ v_i \} (k_\gamma \lambda_2(Q))^2}
\\ & \leq 
\frac{W(\gamma(0))}{\min_{i}\{ v_i \}}
\exp^{ - \frac{(k_\gamma \lambda_2(Q))}{2 \max_{i}\{ v_i \}} t }  +  \frac{d_e}{2}.
\end{align*}
The proof is completed by noting that
\begin{align*}
\frac{W(\gamma(0))}{\min_{i}\{ v_i \}}
\exp^{ - \frac{k_\gamma \lambda_2(Q)}{2\max_{i}\{ v_i \}} t } + \frac{d_e}{2} \leq d_e \qquad \forall t \geq T_\gamma. 
\end{align*}
\end{proof}
The following Lemma establishes a key Lyapunov inequality that allows us to prove GUB and GUUB. 
\begin{lemma} \label{propLyapder}
Consider the networked system in \eqref{system_xinew} such that Assumption~\ref{assGPH} holds. 
Then, there exists $\gamma_o > 0$ and $\alpha_o > 1$ such that, for each $\Gamma, \Theta : \mathbb{R}_{\geq 0} \rightarrow \mathcal{D}$ such that  
$ \gamma_s(t) \geq \gamma_o$ and 
$\alpha_s(t) \geq \alpha_o$ for all $t \geq 0$,  
the function candidate 
\begin{align} \label{eqLyapcandid}
 W(x) :=  \sum_{i=1}^{N} v_{i} |x_i|^2,  
\end{align}
with $v := (v_{1},v_{2}, ...,  v_{N}) \in \text{Ker}(L^\top)$, satisfies
\begin{equation} \label{eqdotW} 
\begin{aligned}
  \dot  W (x) \leq 
- \sum_{i=1}^{N} v_i H_i(x_i)
 - \frac{\gamma_s \alpha_s}{2} 
\lambda_2(Q) |x|^2_{\mathcal{A}_s}, 
\end{aligned}
\end{equation}
where  
\begin{align*}
H_i(x_i) & := 2 \alpha_i \left(|x_{i}|^2 - x_{i}^\top \Pi_{X_i}(x_i) \right) - 2 p_i(x_i)^\top x_i, 
\\
 Q & := L^\top V + V L, ~~ V := \text{diag}(v_1,...,v_N).
\end{align*}
\end{lemma}
\begin{proof} 
For all $x \in \mathbb{R}^{nN}$, along the network dynamics, we obtain 
\begin{align*}
\nonumber 
 \dot  W (x) & = 
-  \sum_{i=1}^{N} v_i \alpha_i \left(|x_{i}|^2 - x_{i}^\top \Pi_{X_i}(x_i) \right) 
\\ & - x^\top [ L^\top \Gamma \Theta V + V \Gamma \Theta L]  x + 2 \sum_{i=1}^{N} v_i  x_{i}^\top p_{i}(x_i). 
\end{align*}
The result follows if we show the existence of $\alpha_o > 1$ and $\gamma_o > 0$ such that 
\begin{align} \label{eqkeyinterm}
x^\top [ L^\top \Gamma \Theta V + V \Gamma \Theta L]  x
\geq \frac{\gamma_s \alpha_s}{2} 
\lambda_2(Q) |x|^2_{\mathcal{A}_s}
\end{align}
for all $\Gamma$, $\Theta \in \mathcal{D}$ with $\gamma_s \geq \gamma_o$ and $\alpha_s \geq \alpha_o$. Since $\Gamma$,  
$\Theta \in \mathcal{D}$, we conclude that 
$$ \Gamma = \gamma_s I_N + \tilde{\Gamma}, \quad
\Theta = \alpha_s I_N + \tilde{\Theta},
$$
for some  
$\tilde{\Gamma}$,  $\tilde{\Theta} \in d_e \mathbb{B}^N$. As a result, 
\begin{align*}
 [L^\top \Gamma \Theta V + V \Gamma \Theta L] & = \gamma_s \alpha_s [L^\top  V + V L] 
\\& + 
\gamma_s  [L^\top \tilde{\Theta} V + V 
\tilde{\Theta} L]
\\ & + 
\alpha_s  [L^\top \tilde{\Gamma} V + V 
\tilde{\Gamma} L] 
\\ & + 
[L^\top \tilde{\Gamma} \tilde{\Theta} V + V \tilde{\Gamma} \tilde{\Theta} L]
\\ & = \frac{\gamma_s \alpha_s}{2} Q + \Pi,
\end{align*}
where 
\begin{align*} 
\Pi := & \frac{\gamma_s \alpha_s}{2} Q + 
\gamma_s  [L^\top \tilde{\Theta} V + V 
\tilde{\Theta} L] + 
\alpha_s  [L^\top \tilde{\Gamma} V + V 
\tilde{\Gamma} L]  
\\ &
+ 
[L^\top \tilde{\Gamma} \tilde{\Theta} V + V \tilde{\Gamma} \tilde{\Theta} L]. 
\end{align*}
 Using \cite[Lemma 1]{lazri:hal-04609468}, we conclude that $Q$ is positive semi-definite and that $\text{Ker}(Q) = \boldsymbol{1}_N$, which implies that 
\begin{align} \label{eqlb}  
 x^\top Q x & = \left[x -  \bold 1_{N}  \bold 1_{N}^\top x/N \right]^\top Q \left[x -  \bold 1_{N}  \bold 1_{N}^\top x/N \right] \\ \nonumber &  \geq  \lambda_2(Q) |x|_{\mathcal{A}_s}^2.  
\end{align}
Hence, it remains to show that $\Pi$ is positive semi-definite when $\alpha_s$ and $\gamma_s$ are sufficiently large. Indeed, we start re-expressing the matrix $\Pi$ as follows
\begin{align*} 
\Pi & := \gamma_s \left[ \alpha_s Q/6 + 
 Q_1 \right] + 
\alpha_s \left[ \gamma_s Q/6 +   Q_2  \right] + 
\gamma_s \alpha_s Q/6 + Q_3, 
\end{align*}
where 
\begin{align*} 
Q_1 & := [L^\top \tilde{\Theta} V + V 
\tilde{\Theta} L], ~~ Q_2 :=  [L^\top \tilde{\Gamma} V + V 
\tilde{\Gamma} L], 
\\
Q_3 & := [L^\top \tilde{\Gamma} \tilde{\Theta} V + V \tilde{\Gamma} \tilde{\Theta} L]. 
\end{align*}
Next, we observe that 
$$ \boldsymbol{1}_N \in \text{Ker}(Q_i) \qquad \forall i \in \{1,2,3\}. $$
Hence, 
\begin{equation}
\label{eqKey}
\begin{aligned}  
x^\top Q_i x & = \left[x -  \bold 1_{N}  \bold 1_{N}^\top x/N \right]^\top Q_i \left[x -  \bold 1_{N}  \bold 1_{N}^\top x/N \right] 
\\ & \leq \lambda_N(Q_i) |x|^2_{\mathcal{A}_s}
\qquad \forall i \in \{1,2,3\}.
\end{aligned}
\end{equation}
Finally, by letting 
$$ \bar{\lambda}_N(Q_i) = \max  \left\{ \lambda_N(Q_i) : \tilde{\Gamma},  \tilde{\Theta} \in d_e \mathbb{B}^N \right\} $$
and combining \eqref{eqKey} with \eqref{eqlb}, we obtain 
\begin{align*} 
\frac{x^\top \Pi x}{ |x|^2_{\mathcal{A}_s}} & \geq \gamma_s \left[ \alpha_s \lambda_2(Q)/6 - 
 \bar \lambda_N(Q_1) \right] 
 \\ &
 + 
\alpha_s \left[ \gamma_s \lambda_2(Q)/6 -   \bar \lambda_N(Q_2)  \right] 
\\ &
+ 
\gamma_s \alpha_s \lambda_2(Q)/6 - \bar \lambda_N(Q_3). 
\end{align*}
Hence,  for 
$$ \alpha_o \geq \max \left\{ 1, \frac{6 \bar \lambda_N(Q_1)}{\lambda_2(Q)} \right\} $$ 
and 
$$ \gamma_o \geq \max \left\{  \frac{6\bar \lambda_N(Q_2)}{\lambda_2(Q)}, \frac{6 \bar \lambda_N(Q_3)}{\lambda_2(Q)} \right\}, $$
we conclude that $\Pi$ is positive semi-definite for all $t \mapsto \gamma_s(t) \geq \gamma_o$ and  for all $t \mapsto \alpha_s(t) \geq \alpha_o$. 
\end{proof}
In the next lemma, we establish GUB and GUUB for the networked system  \eqref{system_xinew}.  
\begin{lemma} \label{lemGUUB}
The solutions to the networked system \eqref{system_xinew}  satisfying Assumptions~\ref{assGPH} and~\ref{Assbd} are GUB and GUUB.
\end{lemma}
\begin{proof}
 We consider the $\gamma_o > 0$ and $\alpha_o > 1$ introduced in Lemma ~\ref{propLyapder}.  Furthermore,  we let $r_{o} > 0$ be arbitrarily fixed such that $|x(0)| \leq r_{o}$.   
\\
\noindent 1) \underline{ GUUB:}   
Using \eqref{eqdotW} in Lemma~\ref{propLyapder} and Assumption~\ref{Assbd}, we  conclude that 
\begin{align} \label{379}
\hspace{-0.3cm} \dot W (x) \leq C -  \frac{\gamma_o \alpha_o}{2} \lambda_2(Q) |x|_{\mathcal{A}_s}^2  \quad \forall x \in \mathbb{R}^{n}, 
\end{align}
where 
$$ C := \sup_{i \in \{1,...,N\}} \sup_{x_i \in \mathbb{R}^n, \alpha_i > 1} \left[- H_i(x_i) \right]. $$
As a consequence,
given $\epsilon >0$,
we conclude that
\begin{align} \label{257}
 \dot W  (x) &  \leq -\epsilon   \qquad \forall x \notin \mathcal{C}, 
\end{align}
where the set $\mathcal{C} \subset  \mathbb{R}^{n}$ is defined as 
\begin{align*}
&\mathcal C := \left\{ x \in \mathbb{R}^{nN} : |x|_{\mathcal{A}_s} \leq  R_c \right \}, ~~
 R_c^2 := \frac{2(C+\epsilon)}{\gamma_o \alpha_o \lambda_2(Q)}.
\end{align*}
Next, we let 
$\bar \rho := \displaystyle\max_{\hspace{-1em}i \in \{1,2,...,N\}} \rho_{i}$, 
where each $\rho_i > 0$ is such that 
$$ H_i(x_i) \geq |x_i|^2  \qquad \forall |x_i| \geq \rho_i.  $$
Also, we let 
$$  \beta := R_c + N \left[  \bar{\rho} + R_c \right].  $$ 
We will establish the existence of $\Psi : \mathbb{R}^{nN} \rightarrow \mathbb{R}$ that is continuous and positive on 
$\mathcal{C} \backslash \beta \mathbb{B}$ such that
\begin{align} \label{eqPsi}
 \dot W (x) \leq  - \Psi (x)  \qquad \forall x \in \mathcal{C} \backslash \beta  \mathbb{B}.
\end{align}
To do so,  we start noting that,  for all $x \in \mathcal{C} \backslash \beta \mathbb{B}$,  
we have  
\begin{align} \label{eqbounds} 
|x| > \beta \quad \text{and} \quad |x|_{\mathcal{A}_s} \leq R_c.   
\end{align}
We will show that for the considered choice of the parameter $\beta$,  for all $x \in \mathcal{C} \backslash \beta \mathbb{B}$,  we have 
\begin{align} \label{eqboundimport} 
\hspace{-0.2cm} |x_{i}| >   \bar{\rho}  \qquad \forall i \in \mathcal{V}.  \end{align}
Indeed, using the identities
\begin{align*}
x  & =   (\bold 1_{N} \otimes I_n) (  (\bold 1_{N} \otimes I_n)^\top  x)/N  
\\ &
+ \left[ x -   (\bold 1_{N} \otimes I_n) (  (\bold 1_{N} \otimes I_n)^\top  x)/N \right],  
\\
 |x|_{\mathcal{A}_s} & = |x -   (\bold 1_{N} \otimes I_n) (  (\bold 1_{N} \otimes I_n)^\top  x)/N|,  
 \end{align*}  
we conclude that
\begin{align} \label{eqbounds1}
| x | \leq |x|_{\mathcal{A}_s}  +  |   (\bold 1_{N} \otimes I_n)^\top  x|. 
\end{align}
Now,  combining \eqref{eqbounds} and \eqref{eqbounds1}, we conclude that,  for all $x \in \mathcal{C} \backslash \beta \mathbb{B}$,  we have 
\begin{equation}
\label{eqbounds2}
\begin{aligned} 
\beta & < | x | \leq |x|_{\mathcal{A}_s}  + |   (\bold 1_{N} \otimes I_n)^\top  x| 
\\ & \leq  R_c + |   (\bold 1_{N} \otimes I_n)^\top  x|. 
\end{aligned}
\end{equation}
As a result, we conclude that 
\begin{align*}
 |   (\bold 1_{N} \otimes I_n)^\top  x|/N >  (\beta - R_c)/ N  \qquad \forall  x \in \mathcal{C} \backslash \beta \mathbb{B}.
\end{align*} 
Next, for all $i \in \mathcal{V}$, we use the fact that  
\begin{align*} 
x_{i} & =   (\bold 1_{N} \otimes I_n)^\top  x/N  + \left( x_{i}  -   (\bold 1_{N} \otimes I_n)^\top  x / N \right) 
\\ & \geq  |  (\bold 1_{N} \otimes I_n)^\top  x| / N - |\left( x_{i}  -   (\bold 1_{N} \otimes I_n)^\top  x/N \right)|
 \end{align*}
to conclude that, for all $x \in \mathcal{C} \backslash \beta \mathbb{B}$,  we have 
\begin{align*}
|x_{i}| & > (\beta - R_c)/ N -R_c =   \bar{\rho} \qquad \forall i \in \mathcal{V}.   
\end{align*}
Now,  since 
\begin{align}
\nonumber 
- \sum_{i=1}^{N} v_{ i} H_{i} (x_{i})   \leq -   \sum_{i=1}^{N} v_{i} |x_{i}|^2  \leq  0 \quad  \forall x \in \mathcal{C} \backslash \beta \mathbb{B}.
\end{align}
As a result,   in view of \eqref{eqdotW},  we conclude that   
\begin{align*}
\dot W (x) & \leq -\sum_{i=1}^{N} v_{i}  |x_{i}|^2 = - \Psi (x) \quad \forall x \in \mathcal{C} \backslash \beta \mathbb{B}. 
\end{align*}
Hence $\Psi$ is continuous and positive on $\mathcal{C} \backslash \beta \mathbb{B}$. 
At this point,  we conclude that
$$
 \dot W (x) \leq  - \min \{ \Psi (x),  \epsilon \} < 0 \qquad \forall x \in \mathbb{R}^{nN}  \backslash
 \beta \mathbb{B}.
$$
The latter is enough to conclude global attractivity and forward invariance of the set 
\begin{align*} 
\mathcal{S}_{\sigma} & := \{ x \in \mathbb{R}^{nN} : W(x) \leq \sigma \}, ~ \sigma  :=  \max \{ W(y) : y \in \beta \mathbb{B} \}.  
\end{align*}
Note that  $W(x) \leq \sigma$ holds when $|x| \leq r := \left[ \min_i \{ v_{i} \} \right]^{-1} \sigma$, defining the ultimate bound. 
Next, we compute an upperbound, $T (r_o,\gamma_o,\alpha_o)$,  on the time that the solutions starting from the set $r_o \mathbb{B}$ 
would take to reach the compact set  $\beta \mathbb{B} \subset \mathcal{S}_{r}$.  For this purpose,  we assume,  without loss of generality,  that $r_o \geq \beta$ and we define
$$ \epsilon_{o} := \min \{ \min \{ \Psi (x),  \epsilon \}  : |x| \geq \beta, ~ x \in \mathcal{S}_{\sigma_{o}} \} > 0,  $$
where 
$\mathcal{S}_{\sigma_{o}}  := \{ x  \in \mathbb{R}^{nN} 
: W(x) \leq \sigma_{o}
\}$ and $\sigma_{o}   := \max \{ W(y) : y \in r_o \mathbb{B} \}$. 
 Therefore, along every solution $t \mapsto x(t)$ starting from 
$x(0) \in r_o \mathbb{B} \backslash \beta \mathbb{B}$, we have 
$\dot W (x(t)) \leq  - \epsilon_{o}$,  up to the earliest time when $x$ reaches $\mathbb{B}_{\beta}$.  For any earlier time, 
we have 
\begin{equation}\label{338}
W(x(t)) \leq  - \epsilon_{o} t + W(x(0)),
\end{equation} 
so we can take $T(r_o,\gamma_o, \alpha_o) = \sigma_o/\epsilon_{o}$. 
Clearly, $T$ depends only on $(r_o,\gamma_o,\alpha_o)$ and $r$ depends only on 
$\gamma_o$ and $\alpha_o$.  
\\
\noindent 2)  \underline{{\it  GUB}: } We already know that 
$$ |x(t)| \leq r(\gamma_o, \alpha_o) \qquad \forall t \geq T(r_o, \gamma_o,\alpha_o).  $$
Hence,  it remains to upperbound the norm of the solution starting from $r_o \mathbb{B}$ over the interval $[0,T]$.  To this end, we reconsider \eqref{379},  which allows us to write 
\begin{align*} 
\hspace{-0.3cm} \dot W (x) &  \leq C  
\qquad \forall x \in \mathbb{R}^{n}. 
\end{align*}
As a consequence, on the interval  $[0, T(r_o,\gamma_o, \alpha_o)]$,  we have 
\begin{align*} 
W(x(t)) & \leq C + W(x(0))  
 \leq  C  + \sigma_o,
\end{align*}  
Hence, $\forall t \geq 0$,  we have 
\begin{align*}
 |x(t)|^2 \leq  R^2 := \left[ \min_i \{ v_{i} \} \right]^{-1} \left( C + \sigma_o + r \right).   
 \end{align*}
\end{proof}

We next establish global practical synchronization.
\begin{lemma} \label{lemGPracS}
Consider the networked system \eqref{system_xinew} such that Assumptions~\ref{assGPH} and~\ref{Assbd} and the GUB and the GUUB properties hold. Then, the global practical synchronization property in (IV) also holds. 
\end{lemma}
\begin{proof}
We start by 
expressing the network \eqref{system_xinew}
in the compact form 
\begin{align} \label{eqNetComp}
\dot{x} =  (\Theta \otimes I_n) f(x)  - (\Gamma \Theta L \otimes I_n) x + p(x), 
\end{align}
where $$ f(x) := \left[ f_1(x_1)^\top ~ ... ~ f_N(x_N)^\top  \right]^\top, $$
and $f_i(x_i) := - \left[x_i - \Pi_{X_i}(x_i) \right]$ for all $i \in \mathcal{V}$.
Furthermore, having  $\Gamma$,  
$\Theta \in \mathcal{D}$, we write  
$$ \Gamma = \gamma_s I_N + \tilde{\Gamma}, ~
\Theta = \alpha_s I_N + \tilde{\Theta}, ~~ \text{for some} ~  
\tilde{\Gamma}, ~ \tilde{\Theta} \in d_e \mathbb{B}^N.
$$
As a result, the networked system \eqref{eqNetComp}
is expressed as
\begin{align} 
\frac{\dot{x}}{\alpha_s} & =  f(x) 
-  \gamma_s \left[ \left(L + \frac{\tilde{\Theta}}{\alpha_s} L \right) \otimes I_n \right] x
+
\left( \frac{\tilde{\Theta}}{\alpha_s}  \otimes I_n \right) f(x) 
\nonumber  \\ & 
- 
 (\tilde{\Gamma} L \otimes I_n) x
-
\left( \frac{\tilde{\Theta} \tilde{\Gamma}}{\alpha_s} L  \otimes I_n \right) x
+ \frac{p(x)}{\alpha_s}. 
\label{eqNetComp1}
\end{align}
At this point, we let 
$ \tilde{L} :=  \frac{\tilde{\Theta}}{\alpha_s} L $
and 
\begin{align*}
F(x) & := f(x) +
\left( \frac{\tilde{\Theta}}{\alpha_s}  \otimes I_n \right) f(x) 
 \\ & 
- 
 (\tilde{\Gamma} L \otimes I_n) x
-
\left( \frac{\tilde{\Theta} \tilde{\Gamma}}{\alpha_s} L \otimes I_n \right) x
+ \frac{p(x)}{\alpha_s}, 
\end{align*}
so that the network's dynamics takes the form 
\begin{equation}
\label{eqNetComp2}
\begin{aligned} 
\frac{\dot{x}}{\alpha_s} & =  F(x) 
-  \gamma_s \left[ (L + \tilde{L}) \otimes I_n \right] x. 
\end{aligned}
\end{equation}
Note that $F$ is uniformly bounded in $\tilde{\Theta}$, $\tilde{\Gamma} \in d_e \mathbb{B}^N$, $\alpha_s \geq \alpha_o > 1$, and $|p| \leq \bar{p}$. Hence, we omit its dependence on these variables. 
Finally, in the new time scale 
$$ \tau(t) := \int^t_{0} \alpha_s(r) dr,  $$
we obtain 
\begin{equation}
\label{eqNetComp3}
\begin{aligned} 
\frac{d x}{d \tau} & =  F(x) 
-  \gamma_s \left[ (L + \tilde{L}) \otimes I_n \right] x. 
\end{aligned}
\end{equation}
At this point, we let 
$$  \hat{v} := \left(I_N + \tilde{\Theta}/\alpha_s \right)^{-1} v.  $$
Note that $\hat{v}^\top (L + \tilde{L}) = 0$ and, for 
$\alpha^*$ sufficiently large, we conclude that 
$$ \hat{v}(t) \in \mathbb{R}^N_{> 0}  ~~ \forall t \geq 0 \quad  \text{if}  \quad \alpha_s(t) \geq \alpha^*  ~~ \forall t \geq 0. $$  
Since $\lambda_1(L) = 0$ has multiplicity one, the matrix $L$ admits the Jordan-block decomposition
$$
L = V 
\begin{bmatrix}
0 & 0 
\\ 
0 & \Lambda
\end{bmatrix} 
V^{-1},
$$
where $\Lambda \in \mathbb{R}^{(N-1) \times (N-1)}$ is composed by the Jordan blocks corresponding to the non-null eigenvalues of $L$. 
Moreover, $V$ and $V^{-1}$ are given by
$$
V := \begin{bmatrix}
    \bold 1_{N} & U  \end{bmatrix}, 
    \qquad V^{-1} := \begin{bmatrix}
    v^\top 
    \\ U^\dagger
\end{bmatrix},
$$ 
where $U \in \mathbb{R}^{N \times N - 1}$ and $U^\dagger \in \mathbb{R}^{N -1 \times N}$ are such that
\begin{align*}
v^\top \bold 1_{N} = 1, \quad 
v^\top U = 0, 
\quad 
U^\dagger \bold{1}_N = 0, \quad 
U^\dagger U = I_{N - 1}.
\end{align*}
Now, we introduce the new coordinates
$$
\begin{bmatrix}
    x_m  \\ e_v 
\end{bmatrix} := \begin{bmatrix}
    (\hat{v}^\top \otimes I_n)  x \\ (U^\dagger \otimes I_n)  x 
\end{bmatrix} := \left( \hat{V}^{-1} \otimes I_n \right) x.
$$
Using the fact that 
\begin{align*}
    V \hat{V}^{-1} & = \bold 1_{N} \hat{v}^\top + U U^\dagger  = I_{N} + \bold 1_{N} 
    (\hat{v} - v)^\top
    \\ & = 
    I_N - \bold 1_{N} (\tilde{\Theta}/\alpha_s) \hat{v}^\top,
    \end{align*}  
we conclude that the state vector $x$ can be expressed as 
\begin{equation}\label{589}
x =  \left( \left(\bold{1}_{N} + \bold{1}_{N} (\tilde{\Theta}/\alpha_s)  \right) \otimes I_n \right)  x_m + (U \otimes I_n) e_v.
\end{equation}
Note that $x_m = (\hat{v}^\top \otimes I_n) x$ corresponds to a 
time-dependent \textit{weighted average} of the network's states.  
As a result, using the new state variable $\bar{x} := (x_m,e_v)$, the network dynamics will consist of three interconnected components:
\begin{subequations}
\label{eqeqeq}
\begin{align}
\frac{d x_m}{d \tau} & = F_m(x_m, e_v), \label{xmdyn} \\
\label{evdyn} 
\frac{d e_v}{d \tau} & = - \gamma_s \left[ (\Lambda + U^\dagger \tilde{L} U ) \otimes I_{n} \right] e_v + G_e(x_m, e_v), 
\end{align}
\end{subequations}
where 
\begin{align*}
& F_m (x_m,e_v) :=  (\hat{v}^\top \otimes I_n) 
\\ &
F \left(
\left( \left(\bold{1}_{N} + \bold{1}_{N} (\tilde{\Theta}/\alpha_s)  \right) \otimes I_n \right)  x_m + (U \otimes I_n) e_v \right) 
+ \\ &  \left( \frac{d \hat{v}}{d \tau} \otimes I_n \right) 
\\ &
\left(
\left( \left(\bold{1}_{N} + \bold{1}_{N} (\tilde{\Theta}/\alpha_s)  \right) \otimes I_n \right)  x_m + (U \otimes I_n) e_v \right), \\
& G_e(x_m,e_v) := 
 (U^\dagger \otimes I_n) \\ & 
 \left[ F\left(
\left( \left(\bold{1}_{N} + \bold{1}_{N} (\tilde{\Theta}/\alpha_s)  \right) \otimes I_n \right)  x_m + (U \otimes I_n) e_v \right) \right].
\end{align*}
In view of \eqref{589}, when $e_v= 0$, we conclude that 
$$x = \left( \left(\bold{1}_{N} + \bold{1}_{N} (\tilde{\Theta}/\alpha_s)  \right) \otimes I_n \right)  x_m, $$ 
thus, 
\begin{equation}
\label{eqAequiv}
\begin{aligned} 
\mathcal{A}_v & := \{ (x_m,e_v) \in \mathbb{R}^{n N} : e_v=0 \} 
\\ &
= \{ x \in \mathbb{R}^{nN} : x_1 = x_2 = ... = x_N \}
= \mathcal{A}_s. 
\end{aligned} 
\end{equation}
In view of \eqref{eqAequiv}, we propose to analyze global practical synchronization 
by analyzing global practical asymptotic stability (GpAS) of the set $\mathcal{A}_v$ for \eqref{eqeqeq}. 
\begin{itemize}
\item GpAS of 
$\mathcal{A}_v$. 
There exist 
$\beta_1$, $\beta_2 \in \mathcal{K L}$ and $\alpha^* > \alpha_o > 1$ such that, given $\varepsilon > 0$, there exist  $\gamma^*(\varepsilon) > \gamma_o > 0$  such
that, for each $\Gamma, \Theta: \mathbb{D}^N \rightarrow  \mathcal{D}$ with 
$\gamma_s(t) \geq \gamma^*$, $\alpha_s(t) \geq \alpha^*$ for all $t \geq 0$,  
and for each $(\bar{x}(0),t) \in \mathbb{R}^{nN} \times \mathbb{R}_{\geq 0}$, we have  
\begin{equation}
\label{eqpropty1C}
\begin{aligned}  
 |\bar{x}(t)|^2_{\mathcal{A}_v} & \leq  \varepsilon + \beta_1(|\bar{x}(0)|_{\mathcal{A}_v}, \alpha^* t)
 \\ &  + \frac{\beta_2(|\bar{x}(0)|, \alpha^* t)}
{\gamma^*}
 \qquad  \forall t \geq 0.
\end{aligned}
\end{equation}
 \end{itemize}
 Given $\gamma_o > 0$,  $\alpha_o > 1$, and the range of initial conditions $r_o \mathbb{B}$, for some $r_o >0$.  When both  GUUB and GUB hold for some $r(\gamma_o, \alpha_o)$, $T(r_o,\gamma_o, \alpha_o)$, and $R(r_o,\gamma_o,\alpha_o)$,  it follows that  
\begin{align} 
\begin{matrix}
|e_v(t)| \leq r_e(\gamma_o,\alpha_o) 
\\
|x_m(t)| \leq r_m(\gamma_o,\alpha_o)
\end{matrix}
 \qquad \forall t \geq T, \label{eqbound1}
\\
\begin{matrix}
|e_v(t)|  \leq R_e(r_o,\gamma_o,\alpha_o) \\ 
|x_m(t)| \leq R_m(r_o,\gamma_o,\alpha_o)
 \end{matrix}
 \qquad \forall t \geq 0, \label{eqbound2}
\end{align}
where 
\begin{align*}
r_e & := |U^\dagger| r, ~
 r_m := |\hat{v}|_\infty r,
~
R_e  := |U^\dagger| R, ~ 
R_m := |\hat{v}|_\infty R.
\end{align*}
Next, since the eigenvalues of $\Lambda$ have positive real parts, we conclude that the origin for the linear system 
$\dot e_v = - (\Lambda \otimes I_n) e_v$ is globally exponentially stable. As a result,  there exists a positive definite matrix $H \in \mathbb{R}^{(N-1)\times (N - 1)}$ such that 
$$ H \Lambda + \Lambda H^\top \geq I_{N-1}. $$
At this point, we let $\alpha^*$ be sufficiently large such that 
$$ 2 H U^\dagger \tilde{L} U \leq I_N/2 \qquad \forall \alpha_s \geq  \alpha^*.  $$
As a result, using the Lyapunov function
$$ V(e_v) :=  e_v^\top H e_v, $$ 
which verifies  
$$
\lambda_{m} |e_v|^2 \leq V(e_v) \leq \lambda_{M} |e_v|^2, 
$$
for  
$(\lambda_m, \lambda_M)  :=   (\lambda_{1} (H),  \lambda_{N -1} (H))$,
we obtain along \eqref{evdyn}:
\begin{equation*}
\begin{aligned} 
& \frac{d V(e_v)}{d \tau}  \leq 
- (\gamma_s/2) |e_v|^2 + 2  \lambda_{M} |G_e(x_m,e_v)| |e_v|.
\end{aligned}
\end{equation*}
In view of the GUB property and \eqref{eqbound2}, we let 
\begin{equation*}
\begin{aligned}
& \bar{G}(r_o,\gamma_o, \alpha_o) \\ & 
:= 
\sup \left\{ |G_e(x_m,e_v)| |e_v| :
\begin{matrix}
|x_m| \leq R_m(r_o,\gamma_o,\alpha_o)
\\
|e_v| \leq R_e(r_o, \gamma_o, \alpha_o)
\end{matrix}
\right\}. 
\end{aligned}
\end{equation*}
As a result, we obtain  
\begin{equation*}
\begin{aligned} 
\frac{d V(e_v)}{d \tau} &  \leq  - \frac{\gamma_s}{2 
  \lambda_{M}} V(e_v) + 2 \lambda_M \bar{G}. 
\end{aligned}
\end{equation*}
Hence, on the interval $[0,T]$, we have 
\begin{align*}
 V(e_v(t)) & \leq \exp^{- \frac{\gamma_s \tau(t)}{2\lambda_{M}}} V(e_v(0)) 
+ 4 \frac{ 1- \exp^{- \frac{\gamma_s \tau(t)}{2 \lambda_{M}}}  }{ \gamma_s} \lambda^2_{\text{M}} \bar{G}. 
\end{align*}
The latter allows us to write, for all $t \in [0,T]$
\begin{align*}
V(e_v(t)) & \leq 
V(e_v(0)) \exp^{- \frac{ \gamma_s \tau(t)}{2\lambda_{M}}}  
+ 4 \frac{ \lambda^2_{\text{M}} \bar{G}}{ \gamma_s}. 
\end{align*}
Hence, for all $t \in [0,T]$, we have 
\begin{equation}
\label{eqneeded1}
\begin{aligned}  
 V(e_v(t)) & \leq     
V(e_v(0)) \exp^{- \frac{ \gamma_s \tau(t)}{2\lambda_{M}}}  
 + 4 \frac{ \lambda^2_{\text{M}} \bar{G} }{ \gamma_s}  \exp^{\frac{\gamma_o (\tau(T)-\tau(t))}{2 \lambda_{M}}}.  
\end{aligned}
\end{equation}
Now, on the interval  $[T,+\infty)$, under GUUB and in view of \eqref{eqbound1}, we let 
\begin{equation*}
\begin{aligned}
 & \hat{G}(\gamma_o,\alpha_o) 
 \\ & := 
\sup \left\{ |G_e(x_m,e_v)| |e_v| :
\begin{matrix}
|x_m| \leq r_m(\gamma_o, \alpha_o)
\\
|e_v| \leq r_e(\gamma_o,\alpha_o)
\end{matrix}
\right\}. 
\end{aligned}
\end{equation*}
As a result, on the interval  $[T,+\infty)$,  we have 
\begin{equation*}
\begin{aligned} 
\frac{d V(e_v)}{d \tau} &  \leq  - \frac{\gamma_s}{2 \lambda_{M}} V(e_v) + 2 \lambda_M \hat{G}. 
\end{aligned}
\end{equation*}
Hence, on the interval $[T,+\infty)$, we have 
\begin{align*}
& V(e_v(t)) 
\\ &
\leq 
\exp^{- \frac{\gamma_s (\tau(t)-\tau(T))}{2 \lambda_{M}}} V(e_v(T)) 
+ \frac{4 \lambda^2_{\text{M}} \hat{G}}{\gamma_s}
\\ & \leq 
   V(e_v(0)) 
\exp^{- \frac{\gamma_s \tau(t)}{2 \lambda_{M}}} +  \frac{ 4 \lambda^2_{\text{M}} \bar{G}}{ \gamma_s}  \exp^{ \frac{\gamma_s (\tau(T)-\tau(t))}{2\lambda_{M}}}  
 + \frac{4 \lambda^2_{\text{M}} \hat{G}}{\gamma_s}
 \\ & \leq 
  V(e_v(0)) 
\exp^{- \frac{\gamma_s \tau(t)}{2 \lambda_{M}}} +  \frac{ 4 \lambda^2_{\text{M}} \bar{G}}{ \gamma_s}  \exp^{ \frac{\gamma_o (\tau(T)-\tau(t))}{2\lambda_{M}}}   
 + \frac{4 \lambda^2_{\text{M}} \hat{G}}{\gamma_s}. 
\end{align*}
The latter allows us to write, for all $t \geq 0$,
\begin{align*}
 V(e_v(t)) & \leq 
\lambda_M |\bar{x}(0)|^2_{\mathcal{A}_v}
\exp^{- \frac{\gamma_o \tau(t)}{2 \lambda_{M}}} 
\\ &
+ \frac{ 4 \lambda^2_{\text{M}} \bar{G}}{ \gamma_s}  \exp^{ \frac{\gamma_o \tau(T)}{2\lambda_{M}}} \exp^{- \frac{\gamma_o \tau(t)}{2 \lambda_{M}}}  
 + \frac{4 \lambda^2_{\text{M}} \hat{G}}{\gamma_s}. 
\end{align*}
Now, given $\varepsilon > 0$,  we choose $\gamma^*(\varepsilon) > 0$ such that
\begin{align} \label{eqtungamma} 
 4 \lambda^2_{\text{M}} \hat{G}(\gamma_o, \alpha_o) / \varepsilon   \leq \gamma^*. 
\end{align}
As a result, for each 
$\gamma_s \geq \gamma^*$, we have 
\begin{equation}
\label{eqneeded2} 
\begin{aligned} 
 \hspace{-0.3cm}  V(e_v(t))  & \leq 
\lambda_M |\bar{x}(0)|^2_{\mathcal{A}_v}
\exp^{- \frac{\gamma_o \tau(t)}{2 \lambda_{M}}} 
\\ &
+ \frac{ 4 \lambda^2_{\text{M}} \bar{G} \exp^{ \frac{\gamma_o \tau(T)}{2\lambda_{M}}}}{ \gamma^*}   \exp^{- \frac{\gamma_o \tau(t)}{2 \lambda_{M}}}  
 + \varepsilon   \quad \forall t \geq 0.  
\end{aligned}
\end{equation}
Finally, combining \eqref{eqneeded1}-\eqref{eqneeded2} and using the fact that $\tau(t) \geq \alpha^* t$, the GpAS property is verified for 
\begin{align*}
\beta_1(a, t) & :=
\frac{\lambda_M a^2}{\lambda_m} \exp^{- \frac{ \gamma_o t}{2\lambda_{M}}},
\\
 \beta_2(r_o, t) & :=  
\frac{4 \lambda^2_{\text{M}} \bar{G} \exp^{\frac{\gamma_o \alpha^*}{2\lambda_{M}} T } }{ \lambda_m}   \exp^{-\frac{ \gamma_o t}{2 \lambda_{M}}}.
\end{align*}
\end{proof}

\begin{lemma} \label{lemadded}
Consider the networked system \eqref{system_xinew} such that Assumptions~\ref{assGPH}-\ref{Assbd}, the GUB, the GUUB, and the global practical synchronization properties hold. Then, the 
practical-optimality property in (V) also holds.  
\end{lemma}

\begin{proof}
We start using the Lyapunov function candidate 
\begin{align}
  Z(x) := \sum^{N}_{i=1} v_i |x_i|^2_{X_o},  
\end{align}
with $v := (v_{1},v_{2}, ...,  v_{N})^\top \in \text{Ker}(L^\top)$.
Along the network's dynamics, we verify
\begin{align*}
   \dot{Z}(x) & = 2 \sum^{N}_{i=1} v_i   \sum_{j = 1}^{N} (\gamma_i \alpha_i a_{ij}) (x_i - \Pi_{X_o}(x_i))^\top (x_j -x_i)  
   \\ &
   -2 \sum^{N}_{i=1}  v_i \alpha_i (x_i - \Pi_{X_o}(x_i))^\top  (x_i - \Pi_{X_i}(x_i))  
   \\ & 
   + 2 \sum^{N}_{i=1} v_i (x_i - \Pi_{X_o}(x_i))^\top p_i(x_i). 
\end{align*}
Using \eqref{proj_ineq_1}, we conclude that 
$$ (x_i - \Pi_{X_o}(x_i))^\top  (x_i - \Pi_{X_i}(x_i)) \geq |x_i|_{{X_i}}. $$
Similarly, under \eqref{proj_ineq_1+}, we conclude that
$$ (x_i - \Pi_{X_o}(x_i))^\top ( \Pi_{X_o}(x_j) - \Pi_{X_o}(x_i)) \leq 0. $$
Hence, 
\begin{align*} 
& (x_i - \Pi_{X_o}(x_i))^\top (x_j -x_i) \leq  \\ &  (x_i - \Pi_{X_o}(x_i))^\top \left[ (x_j - \Pi_{X_o}(x_j)) - (x_i - \Pi_{X_o}(x_i)) \right]. 
\end{align*}
As a result, if we assume that $\alpha_s \geq \alpha_o$ and $\gamma_s \geq \gamma_o$, it would follow that
\begin{align*}
\dot{Z} & \leq - 
\begin{bmatrix}
(x_1 - \Pi_{X_o}(x_1)) 
\\ \vdots \\
(x_N - \Pi_{X_o}(x_N))
\end{bmatrix}^\top
   R
\begin{bmatrix}
(x_1 - \Pi_{X_o}(x_1)) 
\\ \vdots \\
(x_N - \Pi_{X_o}(x_N))
\end{bmatrix}  
 \\ &
 -2 \sum^{N}_{i=1} v_i \alpha_i |x_i|^2_{X_i} + 2 \sum^{N}_{i=1} v_i (x_i - \Pi_{X_o}(x_i))^\top p_i(x_i),   
\end{align*}
where $R := V \Gamma \Theta L + L^\top  V \Gamma \Theta$.
We already shown in the proof of 
\ifitsdraft
Lemma~\ref{propLyapder} 
\else
\cite[Lemma 2]{Appendix--}
\fi 
that 
\begin{align*}
x^\top R x \geq 0 \qquad \forall \gamma_s \geq \gamma_o, ~ \forall \alpha_s \geq \alpha_o.   
\end{align*}
Hence, we conclude that 
\begin{align*}
   \dot{Z} & \leq 
 -2 \sum^{N}_{i=1} v_i \alpha_i |x_i|^2_{X_i} + 2 \bar{p} \sum^{N}_{i=1} v_i |x_i|_{X_o},  
 \\ & = -2 \sum^{N}_{i=1} v_i \left( \alpha_i |x_i|^2_{X_i} -  \bar{p} |x_i|_{X_o} \right).
\end{align*} 
Since $\Theta \in \mathcal{D}$, we conclude that 
\begin{align*}
    \Theta = \alpha_s I_N + \tilde{\Theta} 
\end{align*}
with $\tilde{\Theta} \in d_e\mathbb{B}^N$.
Therefore,  
\begin{align*}
    \dot{Z}(x) \leq - 2 \sum^{N}_{i=1} v_i( \alpha_s |x_i|^2_{X_i}  + \tilde{\Theta}_{ii}  |x_i|^2_{X_i} - \bar{p}|x_i|_{X_o}).
\end{align*}
At this point, we let 
\begin{align*}
S_i(x_i) := \max_{a \in d_e \mathbb{B}} \{ a |x_i|^2_{X_i} - \bar p |x_i|_{X_o} \} \quad   \forall i \in \mathcal{V}, 
\end{align*}
to be able to write 
\begin{align}
    \dot{Z}(x) \leq - 2 \sum^{N}_{i=1} v_i( \alpha_s |x_i|^2_{X_i} + S_i(x_i)).
\end{align}

Next, since $\gamma_s \geq \gamma_o$ and   $\alpha_s \geq \alpha_o$, we use the GUUB property to conclude the existence of $T(|x(0)|, \gamma_o, \alpha_o) \geq 0$ such that   
\begin{align} 
\begin{matrix}
|x(t)| \leq r(\gamma_o,\alpha_o)
 \end{matrix}
  \qquad \forall t \geq T. \label{eqbound2_step2}
\end{align}
As a result, for all $t \geq T$, we verify  
\begin{equation*}
\dot{Z}(x) \leq - 2 \sum^{N}_{i=1} v_i\left( \alpha_s |x_i|^2_{X_i} - \bar{S}_i \right), 
\end{equation*}
where $\bar{S}_i := \sup_{a\in r \mathbb{B}} \left\{|S_i(a)|  \right\}$. 
A this point, we use \cite[Theorem 1]{19921570108} to conclude the existence of $\psi > 1$ such that 
$$ |x|_{X_o} \leq \psi  \max_i \{ |x|_{X_i} \} \qquad \forall x \in \mathbb{R}^n.  $$
Furthermore, we let $\bar{\alpha}(\varepsilon) \geq \alpha_o$ such that
$$  \bar{\alpha} 
 \sum^{N}_{i=1} v_i |x_i|^2_{X_i} - 2 \sum^{N}_{i=1} \bar{S}_i  \geq 0 \quad \forall x \notin \mathcal{A}_o + \frac{\varepsilon}{4 \psi} \mathbb{B}. $$
As a consequence, for all $\alpha_s \geq \bar{\alpha}(\varepsilon)$ and for all $t \geq T$,  
\begin{equation}
\label{Zdot}
    \dot{Z}(x) \leq 
    - \frac{\alpha_s  \min_i\{ v_i \} \varepsilon^2}{4} \qquad \forall x \notin \mathcal{A}_o + \frac{\varepsilon}{4 \psi} \mathbb{B}. 
\end{equation} 
The latter implies that 
$$x(t_1) \in \mathcal{A}_o + \frac{\varepsilon}{4 \psi} \mathbb{B}, \quad \text{for some } t_1 \in \left[T, T+ \frac{4 Z(x(0))}{\bar{\alpha} \min_i\{v_i\} \varepsilon^2} \right], $$ 
since otherwise $t \mapsto Z(x(t))$ would reach zero before the end of the latter interval, which leads to a contradiction.
 Now, we let 
$$ \bar{Z} := \sup \left\{ Z(x) : x \in \mathcal{A}_o + \frac{\varepsilon}{4 \psi} \mathbb{B}  \right\},  $$
which is finite since $\mathcal{A}_o + \frac{\varepsilon}{4 \psi} \mathbb{B}$ is compact and $Z$ is continuous.  

Assume further the existence of $t_2 \geq t_1$ such that 
$$ x(t) \in \mathcal{A}_o + \frac{\varepsilon}{4 \psi} \mathbb{B} \qquad \forall t \in [t_1,t_2] $$
and, for some $t_3 > t_2$,  
$$ t \mapsto x(t) \notin \mathcal{A}_o + \frac{\varepsilon}{4 \psi} \mathbb{B} \qquad \forall  t \in (t_2,t_3]. $$ 
  This means that 
$$ Z(x(t_2)) \leq \bar{Z}, ~~  \dot{Z}(x(t)) \leq - \frac{\alpha_s  \min_i\{ v_i \} \varepsilon^2}{4} \quad \forall t \in [t_2,t_3]. $$
Since otherwise $t \mapsto Z(x(t))$ would reach zero, we must have
$$ t_3 - t_2 \leq T_{\alpha} :=  \frac{4 \bar{Z}}{\alpha_s \min_i \{ v_i \} \varepsilon^2}.  $$
That is, any interval 
$[\underline{t}, \overline{t}] \subset [t_1, + \infty)$
such that
$$ x(\underline{t}) \in  \mathcal{A}_o + \frac{\varepsilon}{4\psi} \mathbb{B}, ~~~ x(t) \notin  \mathcal{A}_o + \frac{\varepsilon}{4\psi} \mathbb{B} \quad \forall t \in (\underline{t}, \overline{t}], $$ 
must verify 
$  \overline{t} - \underline{t} \leq  T_\alpha$. 
In words, after $t_1$, if the solution $x$ leaves $\mathcal{A}_o + \frac{\varepsilon}{4\psi} \mathbb{B}$, it cannot stay away from it for a duration longer than $T_\alpha$, before re-entering again. 

To complete the proof, it is enough to establish the existence $\alpha^{**}(\varepsilon) \geq \bar{\alpha} (\varepsilon)$ such that, for all $\alpha_s \geq \alpha^{**}$, for all $\gamma_s \geq \gamma^{**}(\varepsilon) := \gamma^{*}\left(\frac{\varepsilon}{4 \psi} \right)$,  it holds that 
\begin{align} \label{eqimply} 
& \text{For all } t \geq T_o, \nonumber 
\\ &
\hspace{-0.4cm} x(t) \in  \mathcal{A}_o + \frac{\varepsilon \mathbb{B}}{4 \psi}  \Rightarrow  x(t+s) \in  \mathcal{A}_o + \varepsilon \mathbb{B} ~~ \forall s \in [0, T_\alpha]. 
\end{align}
where 
$$ T_o := T + T_1 + \frac{4 Z(x(0))}{\bar{\alpha} \min_i\{v_i\} \varepsilon^2},  $$ 
and $T_1(r_o,\alpha^*,\gamma^{**})>0$ is such that 
 \begin{align} \label{equse1}
 |x(t)|_{\mathcal{A}_s}  \leq \frac{\varepsilon \mathbb{B}}{4 \psi} \quad \forall t \geq T_1, 
 \end{align}
see property (IV).

To verify \eqref{eqimply}, we introduce the Lyapunov function candidate 
$ d(x) := \max_{i\in \mathcal{V}} |x_i|^2_{X_o}$.
We can show that along the network dynamics,  we have 
\begin{align*} 
\dot{d}(x(t)) & \leq \left( \max_{i\in \mathcal{V}} |x_i(t)|_{X_o} \right) \bar{p} 
\\ & 
\leq q(r) := \sup_{a \in r \mathbb{B}}  \left[ \left( \max_{i\in \mathcal{V}} |a|_{X_o} \right) \bar{p} \right].  
\end{align*}
As a result, we verify 
$$  d(x(t + s)) \leq q(r) T_\alpha +  d(x(t)) \qquad \forall s \in [0, T_\alpha].  $$
Hence, for all $i \in \mathcal{V}$, we have 
$$  |x_i(t + s)|_{X_i} \leq \sqrt{q(r) T_\alpha} +  \sqrt{d(x(t))} \quad \forall s \in [0, T_\alpha].  $$
Furthermore, for $i_o := \text{argmax}_{i\in \mathcal{V}} |x_i|^2_{X_o} $,  we note that 
$$ \sqrt{d(x(t))} =  |x_{i_o}(t)|_{X_o} \leq \psi \max_{j\in \mathcal{V}} |x_{i_o}(t)|_{X_j}.  $$
Next, by letting $j_o := \text{argmax}_{j \in \mathcal{V}} |x_{i_o}|_{X_j}$, we obtain 
$$ \sqrt{d(x(t))} \leq \psi |x_{i_o}(t)|_{X_{j_o}}.  $$
Hence, using \eqref{eqvarproj}, we obtain
\begin{align*} 
\sqrt{d(x(t))} &  \leq  \psi |x_{j_o}(t)|_{X_{j_o}} + \psi  |x_{i_o}(t) - x_{j_o}(t)|.  
\end{align*}
Using \eqref{equse1} and the left-hand side of \eqref{eqimply}, we obtain 
\begin{align*} 
\sqrt{d(x(t))} &  \leq  \varepsilon/2.  
\end{align*}
Finally, by taking $\alpha^{**} := \frac{16 q(r) \bar{Z}}{\varepsilon^4 \min_i\{ v_i\} }$, we conclude that 
$$ \sqrt{q(r) T_\alpha} \leq \varepsilon/2 \qquad \forall \alpha_s \geq \alpha^{**}.   $$
Hence, 
$$  |x_i(t + s)|_{X_i} \leq \varepsilon \qquad \forall s \in [0, T_\alpha].  $$
\end{proof}

\else 

\fi

\end{document}